\documentclass{amsart}
\usepackage{amsmath, amsthm, amssymb, graphicx, tikz, zref}
\usetikzlibrary{arrows}
\usepackage{enumitem}
\newtheorem{theorem}{Theorem}[section]
\newtheorem{corollary}[theorem]{Corollary}
\newtheorem{lemma}[theorem]{Lemma}

\newtheorem{proposition}[theorem]{Proposition}

\newtheorem{definition}[theorem]{Definition}
\newtheorem{remark}[theorem]{Remark}

\makeatletter
\newcommand{\sectionnotoc}[1]{%
  \begingroup
  \let\@tocwrite\@gobbletwo
  \section*{#1}%
  \endgroup}
\makeatother

\usepackage[margin=1in]{geometry}

\usepackage{relsize}
\usepackage{placeins}
 
\usepackage{comment} 

\usepackage{hyperref}
\usepackage{faktor}

\usepackage[
  backend=biber,
  style=numeric,
  sorting=nyt,
  sortcites=true,
  maxbibnames=99
]{biblatex}
\usepackage{tikz-cd}
\usepackage{graphics}

\usetikzlibrary{bbox}

\usepackage{mathabx,epsfig}
\def\acts{\mathrel{\reflectbox{$\righttoleftarrow$}}}

\hypersetup{
    colorlinks,
    linkcolor={red!50!black},
    citecolor={blue!50!black},
    urlcolor={blue!80!black}
}
\newcommand{\R}{\mathbb{R}}

\newcommand{\PP}{\mathbb{P}}\newcommand{\Z}{\mathbb{Z}}\newcommand{\N}{\mathbb{N}}

\usepackage{amsaddr}

\usepackage{stmaryrd}

\title{A Resolution of the Diagonal for Smooth Projective Toric Varieties}
\author{Reginald Anderson}
\address{University of California, Irvine}
\email{reginala@uci.edu, reginald.anderson.cc@gmail.com}
\date{\today}


\usepackage{ragged2e}

\begin{document}
\justifying
\begin{abstract}
The cellular construction of Bayer--Popescu--Sturmfels extends Beilinson's diagonal resolution from projective space to projective toric varieties whose lattice of principal divisors is unimodular. We investigate the smooth projective case in which this lattice condition fails. The periodic arrangement then contains vertices outside the lattice, and the associated finite cellular complex can acquire homology in degree~$0$ and in higher degrees.

We use a floor-function labeling to assign Laurent monomials to vertices in a way compatible with the periodic hyperplane arrangement. A counterexample shows that, without further hypotheses, the undeformed complex need not resolve the diagonal. We therefore impose a symmetry hypothesis on the fan (central symmetry across the origin) under which the construction yields a locally free resolution of $\mathcal{O}_\Delta$ on $X_\Sigma\times X_\Sigma$. We also discuss how a deformation parameter $\epsilon$ should lead to a broader family of resolutions.
\end{abstract}

\maketitle

\tableofcontents

\section{Introduction}
The derived category $\mathrm{D}^b_{\mathrm{Coh}}(X_\Sigma)$ forms part of the B-side of homological mirror symmetry. Beilinson's resolution of the diagonal $L^\bullet$ \cite{Beilinson1978} was crucial in understanding $\mathrm{D}^b_{\mathrm{Coh}}(\PP^n)$, since it yields full strong exceptional collections of line bundles \cite{huybrechts2006fourier}. Via the associated Fourier--Mukai transform with kernel $L^\bullet$, one obtains concrete descriptions of objects of $\mathrm{D}^b_{\mathrm{Coh}}(\PP^n)$ \cite{BondalRepnAssocAlg}.

Projective space is the basic smooth projective toric variety, so one may ask how far Beilinson's construction extends within the toric category. Bayer--Popescu--Sturmfels answered this question for projective toric varieties satisfying their lattice-theoretic unimodularity condition \cite{bayer-popescu-sturmfels}. This requirement is strictly stronger than smoothness: the twice-iterated toric blow-up studied below is smooth even though its principal-divisor matrix is not unimodular. Here, by a nontrivially weighted projective space we mean a weighted projective space that is not isomorphic, as a variety, to ordinary projective space. Such a weighted projective variety is singular, hence never unimodular in this sense.

In this paper we study resolutions of the diagonal beyond the unimodular case, using a periodic hyperplane arrangement and a floor-function vertex labeling. Our main theorem gives a resolution under a central symmetry hypothesis on the fan. Without symmetry we still obtain the expected cokernel after saturation (Proposition~\ref{prop: main}), but additional hypotheses are required to control the remaining homology, as the counterexample below illustrates.

The construction begins with the periodic coordinate hyperplane arrangement $\mathcal{H}\subset\R^n$. We restrict it to the real span $\R L$ of the lattice $L$, translate the slice by a perturbation parameter $\epsilon$, and divide by the additive $L$-action to obtain the finite cell complex $\mathcal{H}_L^\epsilon/L$. We write $(\mathcal{F}^\bullet_{\mathcal{H}_L^\epsilon/L},\partial^\epsilon)$ for the associated cellular complex, whose terms and differential are defined in Section~\ref{sec: deformedResolution}; for $\epsilon=0$ we use the notation $(\mathcal{F}^\bullet_{\mathcal{H}_L/L},\partial)$. The principal result is the following.

We say that a complete fan $\Sigma\subset N_\R$ is \textbf{centrally symmetric} if $\Sigma(1)$ is invariant under $u\mapsto -u$ and its maximal cones come in opposite pairs.

\begin{theorem}\label{thm: MainTheorem1}
Let $X_\Sigma$ be a smooth projective toric variety whose fan $\Sigma$ is centrally symmetric. Then the complex $(\mathcal{F}^\bullet_{\mathcal{H}_L/L},\partial)$ obtained from the periodic arrangement $\mathcal{H}_L$ and the floor-function vertex labeling sheafifies to a locally free resolution of $\mathcal{O}_\Delta$ on $X_\Sigma\times X_\Sigma$.
\end{theorem}

Without additional hypotheses, the undeformed construction can fail. A counterexample can be constructed from a smooth complete toric surface with rays

\[
(1,0),(0,1),(-1,0),(-2,-1),(-1,-1),(-1,-2),(0,-1)
\]

for which the degree-zero homology of the cellular complex has an associated prime $\langle y_0,y_1\rangle$, where $y_i$ denotes the $y$-variable corresponding to the $i$-th ray in the displayed ordering; in particular, $y_0$ and $y_1$ correspond to the rays $(1,0)$ and $(0,1)$, respectively. Hence the complex does not resolve the diagonal. This motivates imposing additional hypotheses, such as central symmetry, or working with a deformation parameter $\epsilon$ admitting a circuit-generic lift.

The methods of \cite{bayer-popescu-sturmfels} run into trouble with extra vertices which appear in the smooth, non-unimodular setting. We use the floor function to give a new way to label vertices in a convex manner, and under the central symmetry hypothesis we show that this yields the desired resolution.

Several recent papers construct resolutions of the diagonal for toric varieties from viewpoints different from the one used here: Hanlon--Hicks--Lazarev draw on techniques related to wrapped Fukaya categories \cite{hanlon2024resolutionstoricsubvarietiesline}, Brown--Erman give a semigroup-based treatment \cite{Brown_2024}, and Favero--Huang develop homotopy path algebras associated to a stratification of the real torus \cite{favero2022homotopypathalgebras}. The related toric Deligne--Mumford stack paper \cite{anderson2023resolutiondiagonaltoricdelignemumford} treats global quotients $[X_\Sigma/\mu]$ by finite abelian groups and constructs the diagonal object locally using Morita theory, whereas the present paper studies smooth projective toric varieties and analyzes the extra vertices of $\mathcal{H}_L$ by means of floor-function labels and the deformation parameter $\epsilon$. Unlike a construction determined only by the one-skeleton, the cellular labels used here also record how the rays assemble into maximal cones. This additional fan data enters both the virtual resolution statement in Proposition~\ref{prop: main} and the locally free resolution obtained under the symmetry hypothesis in Theorem~\ref{thm: MainTheorem1}.


\section*{Outline}
In Section~2, we recall background for coherent sheaves on a smooth projective toric variety and the construction of Bayer--Popescu--Sturmfels in the unimodular case. In Section~3, we introduce the floor-function labeling and the deformed complexes $(\mathcal{H}_L^\epsilon/L,\partial^\epsilon)$, and we prove that the induced cokernel agrees with the diagonal after saturation by the irrelevant ideal (Proposition~\ref{prop: main}). In Section~4 we give explicit examples. In Section~5 we prove Theorem~\ref{thm: MainTheorem1} under the central symmetry hypothesis by controlling the remaining homology.

\section{Preliminaries}
\label{sec:3.3}

We use Cox's dictionary between coherent sheaves on a toric variety and graded modules over its homogeneous coordinate ring. For a simplicial toric variety $X_\Sigma$ with Cox ring $S$, every coherent sheaf is represented by the sheafification of a finitely generated graded $S$-module, and when $X_\Sigma$ is smooth a finitely generated graded $S$-module $F$ sheafifies to zero exactly when it is killed by a power of the irrelevant ideal $I_{irr}$ \cite[Proposition 3.3 - Theorem 3.11]{coxhgscoordinatering}. We will use this vanishing criterion throughout: after constructing a complex of graded $S$-modules, it is enough for the sheafified complex to resolve $\mathcal{O}_\Delta$ that the remaining homology modules are $I_{irr}$-torsion.

The constructions below produce a cellular complex $(\mathcal{F}^\bullet_{\mathcal{H}_L^\epsilon/L},\partial^\epsilon)$ which always has the expected cokernel after saturating by $I_{irr}$ (Proposition~\ref{prop: main}). Under the central symmetry hypothesis, we show that the remaining homology is also $I_{irr}$-torsion, so the sheafification of the complex yields a locally free resolution of $\mathcal{O}_\Delta$.

The construction used later is built from the cellular resolutions of Lawrence ideals developed by Bayer--Sturmfels and Bayer--Popescu--Sturmfels \cite{bayer-sturmfels,bayer-popescu-sturmfels}. We first fix the lattice notation. If $M$ is the character lattice of the dense torus of $X_\Sigma$, the divisor-class sequence is \cite[Theorem 4.1.3]{C-L-S}
\begin{equation}\label{eq: fundexactseq}
0 \rightarrow M \stackrel{B}{\rightarrow} \Z^{|\Sigma(1)|} \stackrel{\pi}{\rightarrow} \mathrm{Cl}(X_\Sigma) \rightarrow 0.     
\end{equation}  
We write $L=\operatorname{Im}(B)=\ker(\pi)$ for the lattice of principal divisors.
\begin{definition} Following \cite{bayer-popescu-sturmfels}, suppose that $B$ has full column rank. We call $L=\operatorname{Im}(B)$ \textbf{unimodular} when all maximal minors of $B$ are contained in $\{0,+1,-1\}$, and by convention use the same adjective for $X_\Sigma$. \end{definition} 
Here ``unimodular'' always refers to this condition on the matrix $B$, rather than merely to smoothness of the fan. When the condition holds, the Bayer--Popescu--Sturmfels complex resolves the diagonal ideal; after sheafification, its terms become sums of line bundles on $X_\Sigma\times X_\Sigma$. The purpose of the later sections is to alter that construction when $X_\Sigma$ is smooth but $B$ is not unimodular.

Choose an identification of the cocharacter lattice with $N\cong\Z^m$, and write the primitive ray generators as $\textbf{b}_1,\ldots,\textbf{b}_n$. With $B$ the $n\times m$ matrix whose rows are the $\textbf{b}_i$, and with $D_i$ the invariant divisor belonging to the $i$-th ray, the sequence above becomes
\begin{align} 0 \rightarrow \Z^m \stackrel{B}{\rightarrow} \Z^n \stackrel{\pi}{\rightarrow} \mathrm{Cl}(X) \rightarrow 0, \end{align}\label{eq:toric.fund.seq.unimod} 
where $\pi(e_i)=[D_i]$. 

\begin{definition}\label{def: jl} For a sublattice $L\subseteq\Z^n$, its \textbf{Lawrence ideal} is

\[ J_L = \left<\textbf{x}^\textbf{a}\textbf{y}^\textbf{b} - \textbf{x}^\textbf{b}\textbf{y}^\textbf{a} \text{ }|\text{ }\textbf{a}-\textbf{b} \in L \right> \subset S = \Bbbk[x_1, \dots, x_n, y_1, \dots, y_n ]. \]

As usual, $\textbf{x}^\textbf{a}=x_1^{a_1}\cdots x_n^{a_n}$ for $\textbf{a}=(a_1,\ldots,a_n)\in\N^n$. \end{definition}

Let
\[R=\Bbbk[x_\rho\mid\rho\in\Sigma(1)]\cong\Bbbk[x_1,\ldots,x_n]\]
be the Cox ring of $X_\Sigma$, graded by $\mathrm{Cl}(X_\Sigma)$ through \eqref{eq:toric.fund.seq.unimod}. For the product we use
\[ S=R\otimes_{\Bbbk}R=\Bbbk[x_1,\ldots,x_n,y_1,\ldots,y_n], \]
with its $\mathrm{Cl}(X)\times\mathrm{Cl}(X)$ grading. Define
\[
\psi:S\longrightarrow\Bbbk[\mathrm{Cl}(X)]\otimes R,
\qquad
\textbf{x}^{\textbf{u}}\textbf{y}^{\textbf{v}}
\longmapsto
[\textbf{u}]\otimes\textbf{x}^{\textbf{u}+\textbf{v}},
\]
and set $I_X=\ker(\psi)$. The following proposition identifies this homogeneous ideal with the ideal of the diagonal.
\begin{proposition} The diagonal embedding $X_\Sigma\hookrightarrow X_\Sigma\times X_\Sigma$ is represented in the Cox ring $S$ by

\[I_X = \left< \textbf{x}^\textbf{u}\textbf{y}^{\textbf{v}} - \textbf{x}^\textbf{v} \textbf{y}^\textbf{u} \text{ }|\text{ } \pi(\textbf{u}) = \pi(\textbf{v}) \text{ in } \mathrm{Cl}(X) \right> \subset S.\]  \end{proposition}

\begin{proof}
Recall that for $V = Spec(A)$ an affine variety, the diagonal mapping $\Delta: V \rightarrow V \times V$ corresponds to the $\Bbbk$-algebra homomorphism $A \otimes A \stackrel{\phi}{\rightarrow} A$ given by $\sum a_1 \otimes a_2 \mapsto \sum a_1 a_2$, from the universal property of $V \times V$. Next, locally, for $\Sigma \subset N_\R$, the diagonal map $U_i \stackrel{\Delta}{\hookrightarrow} U_i \times U_i$ corresponds to
\[ \Bbbk[\sigma_i^\vee \cap M] \otimes_\Bbbk \Bbbk [\sigma_i^\vee \cap M] \stackrel{\phi}{\rightarrow} \Bbbk [\sigma_i^\vee \cap M ] \] given by $\sum a_1 \otimes a_2 \mapsto \sum a_1 a_2.$ A choice of $\sigma\in \Sigma$ gives the monomial \[ x^{\hat{\sigma} } = \prod_{\rho\not\in\sigma(1) } x_\rho \in R \] for $\sigma$ a maximal cone of $\Sigma$, and $\sigma(1)$ the 1-cones of $\sigma$ \cite{C-L-S}. The map 
\[ \chi^m \mapsto x^{<m>} = \prod_\rho x_\rho^{<m, u_\rho> } \] for $m\in M$ and $u_\rho$ a primitive ray generator of the ray $\rho \in \Sigma(1)$
induces the isomorphism

\begin{align*}
    \pi_\sigma^*: \Bbbk[\sigma^\vee \cap M] &\stackrel{\cong}{\rightarrow} ( R_{x^{\hat{\sigma}}})^G \text{ for }G \text{ given by }\mathrm{Cl}(X) \cong Hom(G, \Bbbk^*) \\
        &\cong (R_{x^{\hat{\sigma}} } )_0. \end{align*}

Here, $(R_{x^{\hat{\sigma}} } )_0$ are elements of degree $0$ in the $\mathrm{Cl}(X)$-grading on $R$. So locally, we have that $(I_\Delta)_{(x^{\hat{\sigma}} \otimes_k x^{\hat{\sigma}})}$ is the kernel of the map
\begin{align*}
(R_{x^{\hat{\sigma}}})_0 \otimes_k ( R_{x^{\hat{\sigma}} })_0 & \stackrel{\phi}{\rightarrow} (R_{x^{\hat{\sigma}}})_0 \text{ by }\\
\sum \textbf{x}^\textbf{u} \otimes_k \textbf{x}^\textbf{v} &\mapsto \sum \textbf{x}^{\textbf{u}+\textbf{v}} \end{align*} for $I_\Delta$ the $\mathrm{Cl}(X \times X)$-homogeneous ideal of $S$ corresponding to the diagonal, and $\textbf{x}$ local coordinates in $U_\sigma$. It is evident that 
\begin{align*}
I_X &= \left< \textbf{x}^\textbf{u} \otimes \textbf{x}^\textbf{v} - \textbf{x}^\textbf{v} \otimes \textbf{x}^\textbf{u} \text{ }|\text{ } \pi(\textbf{u})=\pi(\textbf{v}) \in \mathrm{Cl}(X) \right> \end{align*} from which the proposition follows from the observation that $$(I_X)_{(x^{\hat{\sigma}} \otimes_k x^{\hat{\sigma}})}=(I_\Delta)_{(x^{\hat{\sigma}} \otimes_k x^{\hat{\sigma}})}$$ for $X_\Sigma$ smooth. Smoothness is used here to identify the Cox localizations $(R_{x^{\hat{\sigma}}})_0$ with the coordinate rings of the affine toric charts $U_\sigma$ in the usual geometric quotient presentation and to conclude from equality on the affine cover $\{U_\sigma\times U_\sigma\}$ that the corresponding homogeneous ideals define the same closed subscheme after saturation.  
 \end{proof} 

\begin{corollary} Let $L=\operatorname{Im}(B)=\ker(\pi)$ be the lattice of principal divisors in~\eqref{eq:toric.fund.seq.unimod}. Under the identification $S=\Bbbk[x_1,\ldots,x_n,y_1,\ldots,y_n]$, the diagonal ideal $I_X$ is the Lawrence ideal $J_L$ associated to $L$.  \end{corollary} 

In the notation introduced below, the same ideal is $I_{\Lambda(L)}$, where
\[
\Lambda(L)=\{(u,-u)\mid u\in L\}\subset\Z^{2n}
\]
is the Lawrence lifting. Since $\Lambda(L)\cong L$ as a lattice, unimodularity passes from $L$ to $\Lambda(L)$. We now record the module-theoretic conventions used in the cellular construction.

\begin{definition} Set $R=\Bbbk[\N^n]=\Bbbk[x_1,\ldots,x_n]$ and
\[
T=\Bbbk[\Z^n]=\Bbbk[x_1^{\pm1},\ldots,x_n^{\pm1}].
\]
The inclusion $\N^n\hookrightarrow\Z^n$ makes the Laurent polynomial ring $T$ an $R$-module. \end{definition}

\begin{definition}\cite[p. 1]{bayer-sturmfels} An $R$-submodule $M\subseteq T$ is a \textbf{monomial module} if it is generated by Laurent monomials $\textbf{x}^\textbf{a}$ with $\textbf{a}\in\Z^n$. \end{definition}

\begin{definition} Suppose that the unique minimal set of Laurent-monomial generators of a monomial module $M$ is closed under multiplication and inversion. We then call $M$ a \textbf{lattice module}. In this situation there is a lattice $L\subset\Z^n$ with $L\cap\N^n=\{\textbf{0}\}$ such that
\[ M_L = R\{\textbf{x}^\textbf{a} \text{ }|\text{ }\textbf{a} \in L \} = \Bbbk\{\textbf{x}^\textbf{b} \text{ }|\text{ } \textbf{b} \in \N^n + L \} \subset T. \]
\end{definition}

\begin{definition}\cite[p. 1]{miller-sturmfels} \label{defn: ILDefn} For a lattice $L\subseteq\Z^n$, its $\Z^n/L$-graded \textbf{lattice ideal} is
\[ I_L = \left< \textbf{x}^\textbf{a} - \textbf{x}^\textbf{b} \text{ }|\text{ } \textbf{a}-\textbf{b} \in L \right> \subset R. \]
\end{definition}  


\begin{definition}\cite[Section 3]{bayer-sturmfels} Let $S[L]$ denote the $S$-subalgebra of $S[z_1^{\pm1},\ldots,z_n^{\pm1}]$ spanned by the monomials
$\textbf{x}^{\textbf{a}_1}\textbf{y}^{\textbf{a}_2}\textbf{z}^{\textbf{b}}$ with $(\textbf{a}_1,\textbf{a}_2)\in\N^{2n}$ and $\textbf{b}\in L$. Equivalently, $S[L]$ is the group algebra of $L$ over $S$. 
\end{definition}

\begin{definition} The \textbf{periodic coordinate arrangement} in $\R^n$ is
\[
\mathcal{H}=\bigcup_{i=1}^n\bigcup_{j\in\Z}\{p\in\R^n\mid p_i=j\}.
\]
Thus $\mathcal{H}$ consists of all integral translates of the coordinate hyperplanes. \end{definition}

\begin{definition} Put $\R L=L\otimes_\Z\R$. The arrangement induced on this real span is
\[
\mathcal{H}_L=\R L\cap\mathcal{H}\subseteq\R^n.
\]
\end{definition}

The hull construction need not be minimal, but in the unimodular Lawrence setting Bayer--Popescu--Sturmfels identify a minimal cellular resolution supported on the quotient of $\mathcal{H}_L$ by $L$ \cite{bayer-popescu-sturmfels}. We use the following form of their result.

\begin{theorem} \cite[Theorem 3.5]{bayer-popescu-sturmfels} If $L$ is unimodular, then $\mathcal{H}_L/L$ supports the minimal $S$-free resolution of $J_L$. \end{theorem}

The combinatorial role of unimodularity is that the vertices of $\mathcal{H}_L$ are exactly the points of $L$; a non-unimodular lattice produces additional vertices. In the quotient, a lattice vertex is labeled by the corresponding Laurent monomial in $M_{\Lambda(L)}$, and the labels on higher-dimensional cells determine the cellular differential.

\begin{definition} \label{def: monlabel}
  The monomial label at the vertex $(p_1,\dots, p_n)$ is $$\frac{x^p}{y^p} = \frac{x_1^{p_1} \cdots x_n^{p_n}}{y_1^{p_1} \cdots y_n^{p_n}}.$$ \end{definition}

Each face $F$ is labeled by the Laurent lowest-common-multiple (given by the coordinate-wise maximum) of the monomial labels on the vertices of $F$; we denote this label by $m_F$. For each $i\geq 0$, let $(\mathcal{H}_L/L)_i$ denote the set of oriented $i$-dimensional faces, and let $e_F$ denote the basis element associated to $F\in(\mathcal{H}_L/L)_i$. The $i$-th term of the $(S-\text{Mod})^{\mathrm{Cl}(X\times X)}$ free resolution of $\faktor{S}{J_L}$ is

\[
\mathcal{F}_i\left(\mathcal{H}_L/L\right)
=
\bigoplus_{F\in(\mathcal{H}_L/L)_i}
S(-|m_F|)\,e_F,
\]
where $|m_F|$ denotes the multidegree of $m_F$. For $F\in(\mathcal{H}_L/L)_i$, the differential is

\[
\partial_i(e_F)
=
\sum_{F'\lessdot F}
\operatorname{inc}(F',F)\,
\frac{m_F}{m_{F'}}\,e_{F'},
\]
where $F'\lessdot F$ means that $F'$ is a codimension-one face of $F$, and $\operatorname{inc}(F',F)\in\{\pm1\}$ is the incidence sign determined by the chosen orientations.

Bayer--Popescu--Sturmfels work out the case of $\PP^2$ in \cite[Section 2]{bayer-popescu-sturmfels}. We instead record the cellular resolution of $J_L$ for the smooth unimodular surface $Bl_p\PP^2$. 

\subsection{Example of cellular resolution for $\mathrm{Bl}_p \PP^2$} For $X_\Sigma = \mathrm{Bl}_p \PP^2$ with fan $\Sigma \subset N_\R$ with $\Sigma(1) = \{(1,0), (0,1), (-1,1),(0,-1) \}$, we have the fundamental exact sequence from \eqref{eq:toric.fund.seq.unimod} is 

\[ 0 \rightarrow \Z^2 \stackrel{B}{\rightarrow} \Z^4 \stackrel{\pi}{\rightarrow} \mathrm{Cl}(X_\Sigma) \rightarrow 0 \] with $B = \left[\begin{matrix} 1 & 0 \\ 0 & 1 \\ -1 & 1 \\ 0 & -1 \end{matrix}\right]$ and \begin{align*} \mathrm{Cl}(X_\Sigma) &= \faktor{\left<D_1, \dots, D_4\right>}{\left< D_1 \sim D_3, D_4 \sim D_2 + D_3 \right>}\\ &\cong \left<D_2, D_3\right>, \end{align*} so we choose the basis $\{D_2, D_3\}$ for $\mathrm{Cl}(X_\Sigma)$. We have $\R L \cap \mathcal{H}_L$ given in Figure~\hyperref[fig: RLHL]{1}.

\FloatBarrier
\begin{figure}[h] \label{fig: RLHL}
\begin{tikzpicture}[
  x=0.98cm,
  y=1.00cm,
  line cap=round,
  line join=round,
  every node/.style={font=\small},
  >=stealth
]
  \draw[black,semithick] (0,-3.15) -- (0,4.05);
  \draw[black,semithick] (1,-3.15) -- (1,4.05);
  \node[anchor=south east] at (-0.08,4.15) {$x_1=0$};
  \node[anchor=south west] at (1.08,4.15) {$x_1=1$};
  \draw[black,->,semithick] (0.18,3.38) -- (0.82,3.38);

  \draw[blue,semithick] (-6.15,1) -- (6.15,1);
  \draw[blue,semithick] (-6.15,0) -- (6.15,0);
  \node[left=5pt] at (-6.15,1) {$x_2=1$};
  \node[left=5pt] at (-6.15,0) {$x_2=0$};
  \draw[blue,->,semithick] (-5.60,0.14) -- (-5.60,0.86);

  \draw[red,semithick] (-3.20,-3.20) -- (4.20,4.20);       
  \draw[red,semithick] (-4.20,-3.20) -- (3.20,4.20);       
  \node[anchor=south] at (3.25,4.52) {$x_3=1$};
  \node[anchor=south] at (4.75,4.52) {$x_3=0$};
  \draw[red,->,semithick] (-2.95,-2.86) -- (-3.38,-2.43);

  \draw[green!60!black,semithick] (-5.15,0) -- (6.15,0);
  \draw[green!60!black,semithick] (-6.15,-1) -- (6.15,-1);
  \node[right=5pt] at (6.15,0) {$x_4=0$};
  \node[right=5pt] at (6.15,-1) {$x_4=1$};
  \draw[green!60!black,->,semithick] (5.45,-0.14) -- (5.45,-0.86);
\end{tikzpicture}
\caption{$\R L \cap \mathcal{H}_L$ for $X_\Sigma$}
\end{figure}

\FloatBarrier

\begin{figure}[h]
\label{fig: RLHLmodL}
\centering
\resizebox{0.7\textwidth}{!}
{
\begin{tikzpicture}[
  x=1.10cm,
  y=1.10cm,
  line cap=round,
  line join=round,
  every node/.style={font=\small},
  >=stealth
]
  \coordinate (v00) at (0,0);
  \coordinate (v10) at (8,0);
  \coordinate (v11) at (8,8);
  \coordinate (v01) at (0,8);

  \draw[red,semithick] (v00) -- (v11);

  \draw[blue,semithick] (v01) -- (v11);
  \draw[green!60!black,semithick] (v00) -- (v10);
  \draw[black,semithick] (v00) -- (v01);
  \draw[black,semithick] (v10) -- (v11);

  \draw[blue,->,semithick]
    (3.35,8) -- (4.25,8);
  \draw[green!60!black,->,semithick]
    (3.85,0) -- (4.75,0);

  \draw[black,->,semithick]
    (0,3.10) -- (0,3.82);
  \draw[black,->,semithick]
    (0,3.38) -- (0,4.10);
  \draw[black,->,semithick]
    (8,3.10) -- (8,3.82);
  \draw[black,->,semithick]
    (8,3.38) -- (8,4.10);

  \draw[red,->,semithick]
    (4.55,4.55) -- (3.85,3.85);
  \draw[red,->,semithick]
    (4.25,4.25) -- (3.55,3.55);
  \draw[red,->,semithick]
    (3.95,3.95) -- (3.25,3.25);

  \fill (v00) circle (2.2pt);
  \fill (v10) circle (2.2pt);
  \fill (v11) circle (2.2pt);
  \fill (v01) circle (2.2pt);

  \node[anchor=east] at (-0.22,0)
    {$1$};

  \node[anchor=west] at (8.22,0)
    {$\displaystyle\frac{x_1y_3}{x_3y_1}$};

  \node[anchor=west] at (8.22,8)
    {$\displaystyle
      \frac{x_1x_2y_4}{x_4y_1y_2}$};

  \node[anchor=east] at (-0.22,8)
    {$\displaystyle
      \frac{x_2x_3y_4}{x_4y_2y_3}$};

  \node[anchor=south] at (4.35,8.20)
    {$\displaystyle
      \frac{x_1x_2x_3y_4}{x_4y_2}$};

  \node[anchor=north] at (3.65,7.77)
    {$E_1$};

  \node[anchor=south] at (4.35,0.20)
    {$x_1y_3$};

  \node[anchor=north] at (4.10,-0.25)
    {$E_1$};

  \node[anchor=east] at (-0.28,4.15)
    {$x_2x_3y_4$};

  \node[anchor=west] at (0.25,4.15)
    {$E_2$};

  \node[anchor=east] at (7.78,4.80)
    {$E_2$};

  \node[anchor=west] at (8.30,4.80)
    {$\displaystyle
      \frac{x_1x_2y_3y_4}{y_1}$};

  \node at (4.00,5.20)
    {$E_3$};

  \node at (3.25,4.75)
    {$x_1x_2y_4$};

  \node at (1.15,6.10)
    {$F_2$};

  \node at (2.35,6.10)
    {$x_1x_2x_3y_4$};

  \node at (5.15,3.15)
    {$F_1$};

  \node at (5.65,2.45)
    {$x_1x_2y_3y_4$};
\end{tikzpicture} }
\caption{Fundamental domain for
  $\faktor{(\R L \cap \mathcal{H}_L)}{L}$}
\label{fig: RLHLmodL}

\end{figure} 

\FloatBarrier

The quotient cellular complex $\faktor{(\R L \cap \mathcal{H}_L )}{L}$ is given in Figure~\hyperref[fig: RLHLmodL]{2} with monomial labelings from $M_{\Lambda(L)}$, which are graded by $\mathrm{Cl}(X_\Sigma)$. The monomial labelings are collected in Table~\hyperref[table: monlab]{1}.

\begin{table}[h]
\label{table: monlab}
\caption{Monomial labelings on $\faktor{(\R L\cap \mathcal{H}_L)}{L}$}
\[ \begin{array}{c|c}
v & 1, \frac{x_1y_3}{x_3y_1}, \frac{x_1x_2y_4}{x_4y_1y_2},\frac{x_2x_3y_4}{x_4y_2y_3} \text{ starting from the origin, in counter-clockwise order}\\
\hline 
E_1 & x_1 y_3, \frac{x_1x_2x_3y_4}{x_4y_2} \text{ from bottom to top}\\
\hline
E_2 & \frac{x_1 x_2 y_3 y_4 }{y_1}, x_2x_3y_4 \text{ from right to left} \\
\hline
E_3 & x_1 x_2 y_4 \\ 
\hline
F_1 & x_1 x_2 y_3 y_4\\
\hline
F_2 & x_1 x_2 x_3 y_4\\
\hline
\end{array}\]
\end{table}

The boundary maps are given by
\begin{align*}
    \partial^{-2} &= \left(\begin{matrix} x_2y_4 & x_4 y_2\\
    y_1 & x_1 \\ y_3 & x_3 \end{matrix}\right)
\end{align*}
and
\begin{align*}
    \partial^{-1} &= \left(\begin{matrix} x_3y_1-x_1y_3, & x_4 y_2 y_3 - x_2 x_3 y_4, & x_1 x_2 y_4 - x_4 y_1 y_2 \end{matrix}\right). 
\end{align*}

\flushleft
Set
\[
\begin{aligned}
F^{-2}={}&S(-|x_1x_2y_3y_4|)\\
&\oplus S(-|x_1x_2x_3y_4|),\\[0.3em]
F^{-1}={}&S(-|x_1y_3|)\\
&\oplus S\left(-\left|\frac{x_1x_2y_3y_4}{y_1}\right|\right)\\
&\oplus S(-|x_1x_2y_4|).
\end{aligned}
\]
Together, these give the $(\Z^{2n}/\Lambda(L))$-graded $S{\text-}$free resolution
\[
0\longrightarrow F^{-2}
\xrightarrow{\partial^{-2}}F^{-1}
\xrightarrow{\partial^{-1}}S
\longrightarrow\faktor{S}{J_L}
\longrightarrow0.
\]

Set
\[
\begin{aligned}
\mathcal{E}^{-2}={}&
\mathcal{O}(-1,-1)\boxtimes\mathcal{O}(-1,-2)\\
&\oplus
\mathcal{O}(-1,-2)\boxtimes\mathcal{O}(-1,-1),\\[0.3em]
\mathcal{E}^{-1}={}&
\mathcal{O}(0,-1)\boxtimes\mathcal{O}(0,-1)\\
&\oplus
\mathcal{O}(-1,-1)\boxtimes\mathcal{O}(-1,-1)\\
&\oplus
\mathcal{O}(-1,-1)\boxtimes\mathcal{O}(-1,-1).
\end{aligned}
\]
Via Cox's theorem, this corresponds to the locally free resolution
\[
0\longrightarrow\mathcal{E}^{-2}
\longrightarrow\mathcal{E}^{-1}
\longrightarrow\mathcal{O}
\longrightarrow\mathcal{O}_\Delta
\longrightarrow0
\]
in $\mathrm{Coh}(X_\Sigma\times X_\Sigma)$.

\section{Resolving the diagonal for smooth, non-unimodular toric varieties in families} \label{sec: deformedResolution}

\justifying

The goal of this section is to construct the deformed cellular complexes and prove Proposition~\ref{prop: main}, which identifies the homological degree~$0$ cokernel after saturation by the irrelevant ideal. The remaining homology is treated in Section~\ref{sec:torsion}. When $X_\Sigma$ is smooth and non-unimodular, the intersection of the infinite hyperplane arrangement $\mathcal{H}_L$ with $\R L$ will have more lattice points than just what appears in $L$ \cite[Proposition 2.2]{bayer-popescu-sturmfels} in the sense that

\[
(\mathcal{H}_L\cap\R L)\setminus L\neq\emptyset.
\]
Choose a class $\epsilon\in\mathrm{Cl}(X_\Sigma)_\R$ and a lift $a=(a_1,\ldots,a_n)\in\R^n$ satisfying $\pi_\R(a)=\epsilon$. We deform the slice $\R L\cap\mathcal{H}_L$ by requiring at least one coordinate of
\begin{align}\label{eqn: deformation}
 z_1 \vec{C_1} + \cdots + z_m \vec{C_m} + \sum_{i=1}^n a_i e_i
\end{align}
to be an integer; the resulting arrangement is denoted $\mathcal{H}_L^\epsilon$. Here, $z_1,\dots,z_m$ are coordinates on $\R L$, $m$ is the rank of $M$, the vectors $\vec{C}_1,\dots,\vec{C}_m$ are the columns of the matrix $B$ from \eqref{eq:toric.fund.seq.unimod}, and $e_1,\dots,e_n$ are the standard basis vectors of $\R^n$. The undeformed arrangement is recovered by taking $a=0$.

\begin{definition}\label{def:circuit-generic}
A vertex of $\mathcal{H}_L^\epsilon$ is called \textbf{transverse} if the normals of all hyperplanes passing through it are linearly independent; equivalently, exactly $m=\dim(\R L)$ hyperplanes meet there. For $1\leq i\leq n$, let $\lambda_i:\R L\to\R$ be the restriction of the $i$-th coordinate function. A \textbf{primitive circuit} is a primitive vector $c=(c_1,\dots,c_n)\in\Z^n$ whose support is minimal among the supports of nonzero integral relations
\[
\sum_{i=1}^n c_i\lambda_i=0.
\]
We call the lift $a$ (and the corresponding deformation) \textbf{circuit-generic} if
\[
\langle c,a\rangle\notin\Z
\qquad\text{for every primitive circuit }c.
\]
\end{definition}

\begin{lemma}\label{lem:circuit-transversality}
Every vertex of $\mathcal{H}_L^\epsilon$ is transverse if and only if the lift $a$ is circuit-generic.
\end{lemma}

\begin{proof}
Write
\[
H_{i,k}(a)=\{z\in\R L\mid \lambda_i(z)+a_i=k\},
\qquad k\in\Z.
\]
Suppose that the hyperplanes indexed by the support of a primitive circuit $c$ meet at a point $z$. Setting $k_i=\lambda_i(z)+a_i\in\Z$ gives
\[
\langle c,a\rangle
=\sum_i c_i k_i-\sum_i c_i\lambda_i(z)
=\sum_i c_i k_i\in\Z.
\]
Conversely, suppose that $\langle c,a\rangle\in\Z$. Since $c$ is primitive, there are integers $k_i$, indexed by the support of $c$, with $\sum_i c_i k_i=\langle c,a\rangle$. The circuit relation is the only linear relation among the corresponding $\lambda_i$, so the equations $\lambda_i(z)+a_i=k_i$ are consistent. An independent subset of these equations can be extended, using the coordinate restrictions $\lambda_j$, to a basis of $(\R L)^*$; assigning integer levels to the additional equations produces a vertex at which the circuit hyperplanes also meet. That vertex is nontransverse. Thus nontransverse vertices occur exactly on the circuit-integrality locus.
\end{proof}

\begin{remark}\label{rem:effectivity-transversality}
The transversality condition is independent of effectivity. Replacing $a$ by $a+\ell$ with $\ell\in\R L$ leaves the affine slice $\R L+a$ unchanged, while replacing $a$ by $a+u$ with $u\in\Z^n$ does not change the arrangement, because the defining conditions are congruences modulo $\Z$. Hence transversality depends only on the image of $\epsilon$ in
\[
\mathrm{Cl}(X_\Sigma)_\R/\mathrm{Cl}(X_\Sigma),
\]
and the nontransverse locus is the finite union of circuit hypertori defined by $\langle c,a\rangle\in\Z$. Membership in the effective cone is therefore not required for transversality. When a nonempty toric GIT phase is needed, we impose separately that $\epsilon$ lie in the relative interior of the desired chamber of the effective cone. For a small perturbation one may choose a direction $\eta\in\R^n$ with $\langle c,\eta\rangle\neq0$ for every primitive circuit and set $a=t\eta$ for sufficiently small $t\neq0$; the image of $\eta$ may be chosen inside any prescribed open secondary-fan chamber.
\end{remark}

Whenever transversality of $\mathcal{H}_L^\epsilon$ is used below, we assume that the chosen lift $a$ is circuit-generic. Proposition~\ref{prop: main} shows that the homological degree~$0$ cokernel agrees with the diagonal after saturating by the irrelevant ideal even when $\epsilon=0$; controlling the remaining homology requires additional hypotheses, as indicated by Remark~\ref{rem:counterexample} and by the argument in Section~\ref{sec:torsion}. To assign a monomial in $M_{\Lambda(L)}$ to each vertex, we take the floor function of each coordinate of every vertex in $\mathcal{H}_L^\epsilon$. These labels specialize to the monomials appearing in Bayer--Popescu--Sturmfels' result when $X_\Sigma$ is unimodular, though the floor-function labeling used here is new.

\begin{definition}\label{def:deformed-complex}
Let $P=(\mathcal{H}_L^\epsilon/L)$ be the finite quotient cell complex, and orient each cell of $P$ once and for all. For each $i\geq 0$, write $P_i$ for the set of oriented $i$-dimensional faces. For a vertex $v$ represented by a point $p=(p_1,\ldots,p_n)\in\mathcal{H}_L^\epsilon(0)$, set
\[
  m_v=\frac{x^{\lfloor p\rfloor}}{y^{\lfloor p\rfloor}}
  =\prod_{i=1}^n x_i^{\lfloor p_i\rfloor}y_i^{-\lfloor p_i\rfloor}
  \in S_{\prod_i x_i y_i}.
\]
For a face $F$ of $P$, define $m_F$ to be the Laurent lowest-common-multiple of the labels $m_v$ over all vertices $v$ of $F$, computed coordinate-wise on exponent vectors. The cellular complex $(\mathcal{F}^\bullet_{\mathcal{H}_L^\epsilon/L},\partial^\epsilon)$ is the complex of free $S$-modules with
\[
  \mathcal{F}_i^\epsilon=\bigoplus_{F\in P_i} S(-|m_F|)e_F,
\]
and differential
\[
  \partial_i^\epsilon(e_F)=\sum_{F'\lessdot F} \operatorname{inc}(F',F)\,\frac{m_F}{m_{F'}}e_{F'},
\]
where $F'\lessdot F$ means that $F'$ is a codimension-one face of $F$, and $\operatorname{inc}(F',F)\in\{\pm 1\}$ is the incidence sign determined by the chosen orientations. The augmentation sends a vertex basis element $e_v$ to its monomial label $m_v$ in $S_{\prod_i x_i y_i}$. When $\epsilon=0$, this is denoted $(\mathcal{F}^\bullet_{\mathcal{H}_L/L},\partial)$.
\end{definition}

When $X_\Sigma$ is unimodular, a convexity argument in \cite[Proposition 1.1]{bayer-sturmfels} shows that $X_{\preceq\mathbf{b}}$, the subcomplex of $\faktor{(\R L\cap\mathcal{H}_L)}{L}$ on the vertices of degree $\preceq\mathbf{b}$, has vanishing reduced homology with coefficients in $\Bbbk$ for every degree $\mathbf{b}\in\Z^n$. For a general deformation, two independent conditions play different roles. Circuit-genericity of the lift $a$ guarantees that the vertices of $\mathcal{H}_L^\epsilon$ are transverse. Separately, placing the class $\epsilon$ in the relative interior of a secondary-fan chamber determines a toric GIT phase $\Sigma_\epsilon$ and its irrelevant ideal. Proposition~\ref{prop: main} shows that the homological degree~$0$ cokernel is torsion with respect to that irrelevant ideal. This assertion alone does not control the kernel of the augmentation or the higher homology of $(\mathcal{F}^\bullet_{\mathcal{H}_L^\epsilon/L},\partial^\epsilon)$; those groups require an additional argument. For the undeformed complex, Section~\ref{sec:torsion} supplies that argument under the central symmetry hypothesis. 
  
An example of the deformed complex $\mathcal{H}_L^\epsilon$ for $Bl_p\PP^2$ is given in Section~\ref{sec: deformed}. 

\subsection{Cokernel of $\partial^\epsilon_0$}
\label{subsec:coker-general}

Here, we consider the map
\begin{align}\label{eqn: mapf} (\mathcal{F}^\bullet_{\mathcal{H}_L^\epsilon}, \partial^\epsilon) \stackrel{f}{\rightarrow} S_{\prod_i x_i y_i} \end{align}
given by $\bigoplus_{v\in \mathcal{H}_L^\epsilon(0)}S(-|m_v|) \rightarrow S_{\prod_i x_i y_i}$ where a vertex in $\mathcal{H}_L^\epsilon(0)$ maps to its vertex monomial label in $M_{\Lambda(L)}$ from Definition~\ref{def: monlabel}. Recall that for $A$ a commutative ring with $1$, we say that an $A-$module $M$ is $I$-torsion for $I$ an ideal of $A$, if and only if there exists a positive integer $k$ such that $I^k M = 0$. For $X_\Sigma$ smooth and projective with no torus factors, tensoring \eqref{eq: fundexactseq} with $\R$ yields
\[
0\longrightarrow L_\R\longrightarrow\R^n
\xrightarrow{\pi_\R}\mathrm{Cl}(X_\Sigma)_\R\longrightarrow0,
\qquad \epsilon=\pi_\R(a).
\]
The effective cone is
\[
\operatorname{Eff}(X_\Sigma)
=\pi_\R(\R_{\geq 0}^n)
=\operatorname{cone}([D_1],\ldots,[D_n]).
\]
Thus the image of an arbitrary $a\in\R^n$ need not be effective: the class $\epsilon=\pi_\R(a)$ lies in $\operatorname{Eff}(X_\Sigma)$ exactly when the coset $a+L_\R$ meets $\R_{\geq0}^n$. For the GIT discussion in this subsection, we therefore assume separately that $\epsilon$ lies in the relative interior of a chamber of the secondary fan contained in the effective cone. A lift $a$ determines a convex piecewise-linear function $g_a:N_\R\to\R$; changing the lift by an element of $L_\R$ adds a global linear function, so the maximal domains of linearity depend only on $\epsilon$. They form a regular fan $\Sigma_\epsilon$, which is constant as $\epsilon$ varies within the chamber. This chamber assumption selects the toric GIT phase and its irrelevant ideal; it is not the condition that makes $\mathcal{H}_L^\epsilon$ transverse. For the general argument below, we assume that the selected phase $X_{\Sigma_\epsilon}$ is smooth. Let
\[
I_\emptyset(\epsilon)=\{\rho_i\in\Sigma(1)\mid \rho_i\text{ is not contained in any maximal cone of }\Sigma_\epsilon\}.
\]
If the Cox ring is still written with variables indexed by the original set $\Sigma(1)$, then the irrelevant ideal for $\Sigma_\epsilon$ in the $x$-variables is
\[
B_{\Sigma_\epsilon}(x)=\left\langle \prod_{\rho_i\notin\sigma(1)}x_i \text{ }\bigg|\text{ } \sigma\in\Sigma_\epsilon(n)\right\rangle,
\]
so the variables with $\rho_i\in I_\emptyset(\epsilon)$ appear in every generator. The irrelevant ideal for $\Sigma_\epsilon\times\Sigma_\epsilon$ is
\[
I_{irr}(\Sigma_\epsilon\times\Sigma_\epsilon)
= B_{\Sigma_\epsilon}(x)\cap B_{\Sigma_\epsilon}(y)
= \left\langle \prod_{\rho_i\notin\sigma(1)}x_i \mid \sigma\in\Sigma_\epsilon(n)\right\rangle
\cap \left\langle \prod_{\rho_i\notin\sigma(1)}y_i \mid \sigma\in\Sigma_\epsilon(n)\right\rangle .
\]
To prove Proposition~\ref{prop: main} below, we require the following Lemmata~\ref{lemma: l1}-\ref{lemma: poly}. Lemma~\ref{lemma: l1} allows us to consider lattice translates of vertices $p\in \mathcal{H}_L^\epsilon(0)$. Lemma~\ref{lemma: seplemma} separates convex sets by a hyperplane. Lemma~\ref{lemma: poly} allows us to work with polynomials, rather than Laurent polynomials by clearing denominators.

\begin{lemma}\label{lemma: l1}
 Given $p\in \mathcal{H}_L^\epsilon(0)$ and $\sigma_1,\sigma_2\in\Sigma_\epsilon(n)$, there exists $k>0$ with $(x^{\hat{\sigma}_1}y^{\hat{\sigma}_2})^k \cdot m_p \in M_{\Lambda(L)}$ if and only if for all $v\in L$ there exists an $\ell>0$ such that $(x^{\hat{\sigma}_1}y^{\hat{\sigma}_2})^\ell \cdot m_{p+v} \in M_{\Lambda(L)}$. \end{lemma}

\begin{proof} Any $v\in L$ has integer coordinates so that $\lfloor v \rfloor = v$ and $m_v = \frac{x^{\lfloor v \rfloor}}{y^{\lfloor v \rfloor }} = \frac{ x^v}{y^v}$. This implies that adding $v\in L$ to $p$ shifts all coordinates by only integer amounts, so that no cancellation occurs when we take a floor function: \[ \lfloor p+v \rfloor = \lfloor p \rfloor + v \]  

In particular, note that if both $p$ and $v$ are in $L$, then $m_{p+v}=m_p \cdot m_v$ since $p, v \in L$ implies $p+v\in L$. So if there exists $k>0$ with $(x^{\hat{\sigma}_1}y^{\hat{\sigma}_2})^k \cdot m_p \in M_{\Lambda(L)}$, then $(x^{\hat{\sigma}_1}y^{\hat{\sigma}_2})^k\cdot m_{p+v} = (x^{\hat{\sigma}_1}y^{\hat{\sigma}_2})^k\cdot m_p m_v \in M_{\Lambda(L)}$ since $M_{\Lambda(L)}$ is closed under the action of $S[L]$. Now closure of $M_{\Lambda(L)}$ under the action of $S[L]$ also proves the reverse direction. \end{proof}

\begin{lemma}\label{lemma: seplemma} \textbf{Separation Lemma} \cite[ (12) in Section 1.2]{fulton}. If $\sigma$ and $\sigma'$ are convex polyhedral cones whose intersection $\tau$ is a face of each, then there is a $u$ in $\sigma^\vee \cap (-\sigma')^\vee$ with \[ \tau = \sigma \cap u^\perp = \sigma' \cap u^\perp. \] \end{lemma}

\begin{lemma}\label{lemma: poly} Given $p\in\mathcal{H}_L^\epsilon(0)$, for all $\sigma_1,\sigma_2\in\Sigma_\epsilon(n)$, there exists $u\in L, k\in \mathbb{N}$ and $\alpha,\beta\in \mathbb{N}^{\Sigma(1)}$ such that $(x^{\hat{\sigma}_1}y^{\hat{\sigma}_2})^km_p = x^\alpha y^\beta m_u \in M_{\Lambda(L)}$. \end{lemma}

\begin{proof}
The lemma holds if and only if there exists $u\in L, k\in \mathbb{N}$ and $\alpha,\beta\in \mathbb{N}^{\Sigma(1)}$ such that

\begin{align*}
(x^{\hat{\sigma}_1}y^{\hat{\sigma}_2})^k m_p m_{-u} &= x^\alpha y^\beta \in S \end{align*}

Let $\lfloor p_{\rho_i}\rfloor := \tilde{a}_{\rho_i}$ for all $i$ so that 

\begin{align*}
\lfloor p \rfloor &= (\lfloor p_{\rho_1} \rfloor, \dots, \lfloor p_{\rho_{\Sigma(1)}}\rfloor)\\
&= (\tilde{a}_{\rho_1}, \dots, \tilde{a}_{\rho_{|\Sigma(1)|}}) \end{align*}

and define the sets $A_+, A_-$ such that 

\begin{align*}
m_p &= \frac{ \prod_{A_+}x_\rho^{a_\rho} \prod_{A_-} y_\rho^{a_\rho}}{\prod_{A_+} y_\rho^{a_\rho} \prod_{A_-} x_\rho^{a_\rho}} \end{align*}
with $a_\rho = |\tilde{a_\rho}|$ for all $\rho$.

Since $X_{\Sigma_\epsilon}$ is smooth, there exists $u_1\in L$ such that $u_1(e_\rho)=a_\rho$ for all $\rho \in (\sigma_1 \cap \sigma_2)(1)=\sigma_1(1)\cap\sigma_2(1)$ (this condition is vacuous if $(\sigma_1\cap\sigma_2)(1)=\emptyset$). Let $\tau:=\sigma_1\cap\sigma_2$.

Applying Separation Lemma~\ref{lemma: seplemma} to $\sigma_1$ and $\sigma_2$ yields an element $u\in \sigma_1^\vee \cap (-\sigma_2)^\vee$ such that $\tau=\sigma_1\cap u^\perp=\sigma_2\cap u^\perp$. Identifying $u$ with its image in $\Z^{|\Sigma(1)|}$ via $B: M\to \Z^{|\Sigma(1)|}$, we obtain $u_2\in L$ such that
\[u_2\big|_{\tau}=0,\qquad u_2\big|_{\sigma_1\setminus \tau}>0,\qquad \text{and}\qquad u_2\big|_{\sigma_2\setminus \tau}<0.\]

Since $x^{\hat{\sigma}_1}y^{\hat{\sigma}_2}$ is a generator of $I_{irr}$, up to torsion by $I_{irr}$ it suffices to consider only factors (over $S$) in the numerator and denominator of $m_p$ for variables $x_\rho$ and $y_\rho$ for $\rho \in \sigma_1\cup \sigma_2$. The idea here is that up to a high enough power of the generator from $I_{irr}$, we can clear denominators such that the only variables remaining in the denominator correspond to rays in $\sigma_1 \cup \sigma_2$. 

Now there exists $k_0 \in \mathbb{N}$ such that \\


\begin{align*}
(x^{\hat{\sigma}_1}y^{\hat{\sigma}_2})^{k_0} m_p &= \frac{ f_1}{\prod_{\sigma_2\cap A_+}y_\rho^{a_\rho} \prod_{\sigma_1 \cap A_-}x_\rho^{a_\rho}} \\
&= \frac{f_2}{\prod_{\sigma_1\cap\sigma_2\cap A_+}y_\rho^{a_\rho} \prod_{(\sigma_2\setminus \sigma_1) \cap A_+}y_\rho^{a_\rho}\prod_{\sigma_1\cap \sigma_2\cap A_-}x_\rho^{a_\rho}\prod_{(\sigma_1\setminus\sigma_2)\cap A_-}x_\rho^{a_\rho}} \end{align*}

by subdivision of sets $A_+, A_-$, for some monomials $f_1, f_2\in S$. Now $m_pm_{-u_1} = m_{p-u_1}$ since $u_1\in L$ gives that there exists $k_1\in \mathbb{N}$ such that 

\begin{align*}
(x^{\hat{\sigma}_1}y^{\hat{\sigma}_2})^{k_1} m_{p-u_1} &= \frac{ f_3}{\prod_{(\sigma_2\setminus \sigma_1)\cap A_+}y_\rho^{b_\rho}\prod_{(\sigma_1\setminus \sigma_2)\cap A_-}x_\rho^{b_\rho}} \end{align*}

for some $b_\rho \in \N$ and some monomial $f_3\in S$, so that there exists $\ell\in \mathbb{N}$ such that $m_{p-u_1}m_{-\ell u_2} = m_{p-u_1-\ell u_2}$ since $\ell u_2\in L$ gives 

\begin{align*}
(x^{\hat{\sigma}_1}y^{\hat{\sigma}_2})^k m_{p-u_1-\ell u_2} &= f_4 \end{align*} for some monomial $f_4 = x^\alpha y^\beta\in S$, for some $\alpha,\beta\in \mathbb{N}^{|\Sigma(1)|}$. 
\end{proof}

\begin{proposition}\label{prop: main}
Assume that $\epsilon$ determines a smooth toric GIT phase $\Sigma_\epsilon$. For the map $f$ from Equation~\ref{eqn: mapf}, the cokernel of $M_{\Lambda(L)}\hookrightarrow\operatorname{Im}(f)$ is torsion with respect to $I_{irr}(\Sigma_\epsilon\times\Sigma_\epsilon)$.
\end{proposition} 

\begin{proof}
Let $p$ be a vertex in $\mathcal{H}_L^\epsilon(0)$. Lemma~\ref{lemma: poly} and Lemma~\ref{lemma: l1} imply that for any choice of $\sigma_1,\sigma_2\in\Sigma_\epsilon(n)$ there exists $k\in\mathbb{N}$ with $(x^{\hat{\sigma}_1}y^{\hat{\sigma}_2})^k\cdot m_p\in M_{\Lambda(L)}$. In particular, every class in $\mathrm{coker}(M_{\Lambda(L)}\hookrightarrow \mathrm{Im}(f))$ is annihilated by a power of the irrelevant ideal $I_{irr}(\Sigma_\epsilon\times\Sigma_\epsilon)$, proving the claim.

For $\epsilon=0$, we use the original fan $\Sigma$ and repeat the argument with $\sigma_1,\sigma_2\in\Sigma(n)$; the cokernel again becomes zero after saturating by $I_{irr}(\Sigma\times\Sigma)$. However, this does \emph{not} by itself imply that $(\mathcal{F}^\bullet_{\mathcal{H}_L/L},\partial)$ resolves the diagonal: one must also control the kernel of the augmentation in homological degree~$0$ and the higher homology modules. As noted in Remark~\ref{rem:counterexample}, additional hypotheses are needed in general. Under the central symmetry hypothesis we complete this step in Section~\ref{sec:torsion}.
\end{proof}

\section{Example: Resolving the diagonal via deformation for $\mathrm{Bl}_p(\PP^2)$ with the deformed complex $\mathcal{H}_L^\epsilon$ }
\label{sec: deformed}

A lift $a\in\R^4$ changes the affine slice used to cut the periodic arrangement, and we write $\epsilon$ for its class in $\mathrm{Cl}(X_\Sigma)_\R$. As $a$ varies, it produces a family of labeled cellular complexes; whenever the corresponding sheafified complexes are exact, they give a family of resolutions of $\mathcal{O}_\Delta$. A circuit-generic perturbation makes the intersections at the vertices of $\mathcal{H}_L^\epsilon$ transverse. We conjecture that this parameter reflects variation of the K\"{a}hler parameter on the mirror side. As an illustrative calculation, we deform $\mathcal{H}_L$ for the smooth Fano and unimodular surface $\mathrm{Bl}_{p}\PP^2$. Let $X_\Sigma$ denote the toric variety associated to the fan with rays $\{(1,0),(1,1),(0,1),(-1,-1)\}$ in $N_\R$. For this fan, the principal-divisor lattice is $L=\operatorname{Im}(B)$ in 


\[ 0 \rightarrow M \stackrel{B}{\longrightarrow} \Z^4 \stackrel{\pi}{\rightarrow}  \mathrm{Cl}(X_\Sigma) \rightarrow 0 \] 

where $B$ is given by $\left[\begin{matrix} 1 & 0 \\ 1 & 1 \\ 0 & 1 \\ -1 & -1 \end{matrix}\right],$ is unimodular. The nef cone in Figure~\hyperref[fig: Bl0p2.nefcone]{3} is given in blue. 

\FloatBarrier
\begin{figure}[h]
\centering
\begin{tikzpicture}[
  scale=0.92,
  transform shape,
  line cap=round,
  line join=round,
  >=stealth,
  every node/.style={font=\small}
]

  \coordinate (O) at (0,0);

  \begin{scope}
    \clip (-0.05,-0.05) rectangle (8.15,6.55);

    %
    \path[
      fill=cyan!70!blue,
      fill opacity=0.34
    ]
      (O)
      -- (90:20)
      arc[
        start angle=90,
        end angle=59,
        radius=20
      ]
      -- cycle;

    %
    \path[
      fill=cyan!70!blue,
      fill opacity=0.34
    ]
      (O)
      -- (55:20)
      arc[
        start angle=55,
        end angle=31,
        radius=20
      ]
      -- cycle;
  \end{scope}

  \draw[->,semithick]
    (O) -- (0:8.15);

  \draw[->,semithick]
    (O) -- (31:7.65);

  \draw[->,semithick]
    (O) -- (90:6.25);

  \draw[->,semithick]
    (0.10,0.16) -- (57:7.55);

  %
  \foreach \r in {0.82,1.22}{

    \draw[thin]
      (3:\r)
      arc[
        start angle=3,
        end angle=28,
        radius=\r
      ];

    \draw[thin]
      (34:\r)
      arc[
        start angle=34,
        end angle=54,
        radius=\r
      ];

    \draw[thin]
      (60:\r)
      arc[
        start angle=60,
        end angle=87,
        radius=\r
      ];
  }

  \fill (O) circle (2.2pt);

  \node[anchor=east] at (-0.10,6.28)
    {$\mathcal{O}(D_1),\mathcal{O}(D_3)$};

  \node[anchor=west] at (6.58,3.88)
    {$\mathcal{O}(D_4)$};

  \node[anchor=west] at (8.18,0)
    {$\mathcal{O}(D_2)$};

  \node[anchor=south west] at (4.05,6.35)
    {$\epsilon=\mathcal{O}(D_1)+\mathcal{O}(D_4)$};

\end{tikzpicture}
\caption{The nef cone for $\mathrm{Bl}_p\PP^2$ }
 \label{fig: Bl0p2.nefcone}
\end{figure}
\FloatBarrier

Here, 

\[ Cl(X_\Sigma) \cong \faktor{\left<D_1, \dots, D_4\right>}{(D_1+D_2-D_4, D_2+D_3-D_4)}\]

Since $D_2$ corresponds to the exceptional divisor in the blow-up, $D_2$ has self-intersection number -1: \[ D_2\cdot D_2=-1\] while $D_1, D_3$, and $D_4$ have self-intersection number $1$, so the nef cone of $X_\Sigma$ is given by the span of $D_1$ and $D_4$. This description of the nef cone comes from a map

\[ (\R^2)^\vee \stackrel{\left[\begin{matrix} 0 & 1 & 0 & 1 \\ 1 & 0 & 1 & 1 \end{matrix}\right]}{\longleftarrow} \R^4 \]

corresponding to a fan in the first orthant of $\R^4$. Fix $0<t<\frac12$ and choose the lift
\[
a=t\left[\begin{matrix}1\\0\\0\\1\end{matrix}\right]\in\R^4,
\]
whose class is $\epsilon=t([D_1]+[D_4])$ in the relative interior of the nef chamber. The chamber condition records the chosen GIT phase; circuit-genericity will be checked separately below. 

\subsubsection{Deforming the cellular complex $\mathcal{H}_L$}

$\R L\cap \mathcal{H}_L$ is given in Figure~\ref{fig:RLHL} for $\mathrm{Bl}_p\PP^2$, which is unimodular (hence no extra vertices appear). Here, vertical lines in yellow correspond to integer values of $x_1$, and $x_1$ increases in value as we move to the right in the diagram. Horizontal lines in black correspond to integer values of $x_3$, which increase as we move up. Diagonal lines in blue give integer values in $x_2$ and $x_4$, which increase up and to the right, and down and to the left, respectively.

Since more than two hyperplanes intersect in $\R L$ at each point, the cellular complex $\mathcal{H}_L$ does not have transverse intersections at each vertex. 

\FloatBarrier
\begin{figure}[h]
\centering
\begin{tikzpicture}[
scale=0.8,
  x=0.85cm,
  y=0.85cm,
  line cap=round,
  line join=round,
  >=stealth
]

  \foreach \xa/\ya/\xb/\yb in {
    -1/1/1/-1,
    -1/3.4/3.4/-1,
    -1/5.8/5.8/-1,
    -1/8.2/8.2/-1,
    1.4/8.2/8.2/1.4,
    3.8/8.2/8.2/3.8,
    6.2/8.2/8.2/6.2
  }{
    \draw[
      blue!85,
      <->,
      line width=0.65pt
    ]
      (\xa,\ya) -- (\xb,\yb);
  }

  \foreach \xcoord in {0,2.4,4.8,7.2}{
    \draw[
      yellow!85!orange,
      <->,
      line width=0.65pt
    ]
      (\xcoord,-1) -- (\xcoord,8.2);
  }

  \foreach \ycoord in {0,2.4,4.8,7.2}{
    \draw[
      black,
      <->,
      line width=0.65pt
    ]
      (-1,\ycoord) -- (8.2,\ycoord);
  }

  \draw[
    yellow!85!orange,
    ->,
    line width=0.8pt
  ]
    (2.45,7.46) -- (2.98,7.46);

  \draw[
    black,
    ->,
    line width=0.8pt
  ]
    (2.88,4.90) -- (2.88,5.44);

  \draw[
    blue!85,
    ->,
    line width=0.8pt
  ]
    (3.22,6.08) -- (2.82,5.68);

  \draw[
    blue!85,
    ->,
    line width=0.8pt
  ]
    (3.58,6.12) -- (3.98,6.52);

\end{tikzpicture}
\caption{$\R L$, showing $\R L \cap \mathcal{H}_L \subset \R^4$}
\label{fig:RLHL}
\end{figure}

To remedy the lack of transverse intersection at each vertex $v\in\mathcal{H}_L(0)$, we translate the slice by the lift $a$ above. The translated affine plane is parameterized by
\[
p_t(x,y)=
x\left[\begin{matrix}1\\1\\0\\-1\end{matrix}\right]
+y\left[\begin{matrix}0\\1\\1\\-1\end{matrix}\right]
+\left[\begin{matrix}t\\0\\0\\t\end{matrix}\right].
\]
The coordinate restrictions are $\lambda_1=x$, $\lambda_2=x+y$, $\lambda_3=y$, and $\lambda_4=-x-y$. The primitive circuits of this configuration may be represented by
\[
c_1=(1,-1,1,0),\qquad
c_2=(0,1,0,1),\qquad
c_3=(1,0,1,1).
\]
Their pairings with $a$ are $t,t,$ and $2t$, respectively. Thus $0<t<\frac12$ makes the lift circuit-generic, independently of the fact that its class lies in the nef chamber. The resulting arrangement $\mathcal{H}_L^\epsilon$ has shifts in the first and fourth coordinate hyperplane families by $t$. We construct the cellular complex $(\mathcal{F}^\bullet_{\mathcal{H}_L^\epsilon},\partial^\epsilon)$ below. The deformed complex $\mathcal{H}_L^\epsilon$ is given in Figure~\ref{fig:HLepsilon}, together with the monomial labelings obtained by applying the floor function coordinate-wise to each vertex in $\mathcal{H}_L^\epsilon\subset\R^4$. Vertices with a common color carry the same monomial label in Figure~\ref{fig:HLepsilon}. The arrows off the hyperplanes indicate the direction in which $x_i$ increases for $1\leq i\leq4$, as in Figure~\ref{fig:RLHL}. 

\FloatBarrier
\begin{figure}
\definecolor{hleC00}{RGB}{0,115,59}
\definecolor{hleC01}{RGB}{250,157,0}
\definecolor{hleC02}{RGB}{255,0,212}
\definecolor{hleC03}{RGB}{241,196,15}
\definecolor{hleC04}{RGB}{207,54,108}
\definecolor{hleC05}{RGB}{150,150,150}
\definecolor{hleC06}{RGB}{148,102,53}
\definecolor{hleC07}{RGB}{81,92,93}
\definecolor{hleC08}{RGB}{164,0,199}
\definecolor{hleC09}{RGB}{20,199,222}
\definecolor{hleC10}{RGB}{169,255,23}
\definecolor{hleC11}{RGB}{18,3,119}
\definecolor{hleC12}{RGB}{237,54,36}
\definecolor{hleC13}{RGB}{96,186,70}
\definecolor{hleC14}{RGB}{98,45,144}
\definecolor{hleC15}{RGB}{0,111,255}
\definecolor{hleC16}{RGB}{23,73,179}
\definecolor{hleC17}{RGB}{0,0,0}
\definecolor{hleC18}{RGB}{226,106,106}
\definecolor{hleC19}{RGB}{181,21,21}
\definecolor{hleC20}{RGB}{203,103,14}
\definecolor{hleC21}{RGB}{47,218,119}
\definecolor{hleC22}{RGB}{25,176,146}
\definecolor{hleC23}{RGB}{157,187,216}
\definecolor{hleC24}{RGB}{52,152,219}
\definecolor{hleC25}{RGB}{68,79,173}
\definecolor{hleC26}{RGB}{199,195,189}
\definecolor{hleC27}{RGB}{176,156,255}
\definecolor{hleC28}{RGB}{140,84,208}
\definecolor{hleC29}{RGB}{208,134,0}
\definecolor{hleC30}{RGB}{244,243,10}
\definecolor{hleC31}{RGB}{165,204,17}

\newcommand{\HLEcluster}[4]{%
  \fill[#4,fill opacity=.36]
    (#1,{#2+.2}) circle (4.2pt);
  \fill[#4,fill opacity=.36]
    (#1,{#2+.4}) circle (4.2pt);
  \fill[#4,fill opacity=.36]
    ({#1+.2},#2) circle (4.2pt);
  \fill[#4,fill opacity=.36]
    ({#1+.4},#2) circle (4.2pt);
  \fill[#3,fill opacity=.36]
    (#1,#2) circle (4.2pt);
}

\centering
\begin{tikzpicture}[
scale = 1.2,
  x=1.55cm,
  y=1.55cm,
  line cap=round,
  line join=round,
  >=stealth,
  every node/.style={font=\scriptsize},
  monlabel/.style={
    draw=gray!45,
    rounded corners=1.5pt,
    fill=white,
    fill opacity=.96,
    text opacity=1,
    inner sep=2.5pt,
    align=center
  },
  callout/.style={
    ->,
    gray!75,
    thin
  },
  increase/.style={
    ->,
    line width=.65pt,
    shorten <=-1.2pt
  }
]

  \foreach \x in {0,1,2,3}{
    \draw[
      yellow!85!orange,
      <->,
      line width=.55pt
    ]
      (\x,-.32) -- (\x,3.55);
  }

  \foreach \y in {0,1,2,3}{
    \draw[
      black,
      <->,
      line width=.55pt
    ]
      (-.32,\y) -- (3.55,\y);
  }

  \foreach \a/\b/\c/\d in {
    -.25/.45/.45/-.25,
    -.25/1.45/1.45/-.25,
    -.25/2.45/2.45/-.25,
    -.25/3.45/3.45/-.25,
    .75/3.45/3.45/.75,
    1.75/3.45/3.45/1.75,
    2.75/3.45/3.45/2.75
  }{
    \draw[
      black,
      <->,
      line width=.55pt
    ]
      (\a,\b) -- (\c,\d);
  }

  \foreach \a/\b/\c/\d in {
    -.25/.65/.65/-.25,
    -.25/1.65/1.65/-.25,
    -.25/2.65/2.65/-.25,
    -.25/3.65/3.65/-.25,
    .75/3.65/3.65/.75,
    1.75/3.65/3.65/1.75,
    2.75/3.65/3.65/2.75
  }{
    \draw[
      blue!85,
      <->,
      line width=.55pt
    ]
      (\a,\b) -- (\c,\d);
  }

  %

  \draw[
    yellow!85!orange,
    increase
  ]
    (1.00,2.72) -- (1.34,2.72);

  \draw[
    black,
    increase
  ]
    (1.78,2.00) -- (1.78,2.34);

  \draw[
    black,
    increase
  ]
    (1.48,2.72) -- (1.73,2.97);

  \draw[
    blue!85,
    increase
  ]
    (1.88,2.52) -- (1.63,2.27);

  \HLEcluster{0}{0}{hleC19}{hleC20}
  \HLEcluster{1}{0}{hleC17}{hleC16}
  \HLEcluster{2}{0}{hleC09}{hleC08}
  \HLEcluster{3}{0}{hleC14}{hleC15}

  \HLEcluster{0}{1}{hleC22}{hleC23}
  \HLEcluster{1}{1}{hleC05}{hleC04}
  \HLEcluster{2}{1}{hleC00}{hleC01}
  \HLEcluster{3}{1}{hleC02}{hleC03}

  \HLEcluster{0}{2}{hleC21}{hleC18}
  \HLEcluster{1}{2}{hleC10}{hleC11}
  \HLEcluster{2}{2}{hleC06}{hleC07}
  \HLEcluster{3}{2}{hleC12}{hleC13}

  \HLEcluster{0}{3}{hleC24}{hleC25}
  \HLEcluster{1}{3}{hleC26}{hleC27}
  \HLEcluster{2}{3}{hleC28}{hleC29}
  \HLEcluster{3}{3}{hleC30}{hleC31}


  \node[
    monlabel,
    draw=hleC07,
    anchor=west
  ] (L01off) at (4.22,3.28)
    {$\displaystyle
      \frac{x_2x_3y_4}{x_4y_2y_3}$};

  \draw[
    ->,
    hleC07,
    thin,
    shorten >=5.4pt
  ]
    (L01off.west) -- (2.00,2.40);

  \node[
    monlabel,
    draw=hleC06,
    anchor=west
  ] (L01grid) at (4.22,2.38)
    {$\displaystyle
      \frac{x_3y_4}{x_4y_3}$};

  \draw[
    ->,
    hleC06,
    thin,
    shorten >=5.4pt
  ]
    (L01grid.west) -- (2.00,2.00);

  \node[
    monlabel,
    anchor=east
  ] (Lm10off) at (-1.05,1.85)
    {$\displaystyle
      \frac{x_4y_1y_2}{x_1x_2y_4}$};

  \draw[callout]
    (Lm10off.east) -- (1.20,1.00);

  \node[
    monlabel,
    anchor=east
  ] (Lm10grid) at (-1.05,1.15)
    {$\displaystyle
      \frac{x_4y_1y_2^2}{x_1x_2^2y_4}$};

  \draw[callout]
    (Lm10grid.east) -- (1.00,1.00);

  \node[
    monlabel,
    anchor=west
  ] (L10off) at (4.10,1.85)
    {$\displaystyle
      \frac{x_1x_2y_4}{x_4y_1y_2}$};

  \draw[callout]
    (L10off.west) -- (3.40,1.00);

  \node[
    monlabel,
    anchor=west
  ] (L10grid) at (4.10,1.15)
    {$\displaystyle
      \frac{x_1y_4}{x_4y_1}$};

  \draw[callout]
    (L10grid.west) -- (3.00,1.00);

  \node[
    monlabel,
    anchor=north
  ] (L0m1grid) at (1.15,-.78)
    {$\displaystyle
      \frac{x_4y_2^2y_3}{x_2^2x_3y_4}$};

  \draw[callout]
    (L0m1grid.north) -- (2.00,0.00);

  \node[
    monlabel,
    anchor=north
  ] (L0m1off) at (3.20,-.78)
    {$\displaystyle
      \frac{x_4y_2y_3}{x_2x_3y_4}$};

  \draw[callout]
    (L0m1off.north) -- (2.40,0.00);

\end{tikzpicture}
\caption{$\mathcal{H}_L^\epsilon$ for $\mathrm{Bl}_p\PP^2$, deformed in the $D_1+D_4$ direction}
\label{fig:HLepsilon}
\end{figure}
\FloatBarrier

Here, we emphasize that $\R L$ is a plane in $\R^4$ along which we are considering the intersections with the hyperplane arrangement $\mathcal{H}_L = \{ (x_1,x_2,x_3,x_4) \text{ }|\text{ } \exists\text{ } 1\leq i \leq 4 \text{ with } x_i \in \Z \}$. Taking the quotient of $\mathcal{H}_L^\epsilon$ by $L$ and giving monomial labels from $M_{\Lambda(L)}$ given by the integer floor function in each component are given in Figure~\ref{fig:HLeQuotL}. Here, the differential 

\[ \partial^\epsilon_1: \bigoplus_{e \in \mathcal{H}_L^\epsilon(1)}S \rightarrow \bigoplus_{v\in \mathcal{H}_L^\epsilon(0)} S \] is given by 

\[\begin{array}{c|cccccccccc|}
v_4 & 0 & 0 & 0 & 0 & 1 & -1 & 0 & x_3y_4 & 0 & -x_3y_1\\ v_3 & 0 & 1 & -y_2 & 0 & 0 & 1 & 0 & 0 & -x_3y_1 & 0 \\ 
v_2 & 0 & 0 & 0 & 1 & -1 & 0 & -x_1y_4 & 0 & 0 & x_1y_3\\
v_1 & y_2 & -1 & 0 & -1 & 0 & 0 & 0 & 0 & x_1y_3 & 0 \\
v_0 & -x_2 & 0 & x_2 & 0 & 0 & 0 & x_4y_1 & -x_4y_3 & 0 & 0 \\
\hline
 & E_0 & E_1 & E_2 & E_3 & E_4 & E_5 & E_6 & E_7 & E_8 & E_9 \end{array} \]

 subject to the labeling in Figure~\ref{fig:HLeQuotL}. The color labeling of vertices and edges in Figure~\ref{fig:HLeQuotL} corresponds to the color labeling in Figure~\ref{fig:HLepsilon}. Exactness away from homological index $i=0$ follows as before for the unimodular case by the same convexity argument used in Theorem 3.1 in Bayer-Popescu-Sturmfels \cite{bayer-popescu-sturmfels}. We investigate the cokernel of $\partial^\epsilon_0$ in the following section.

\begin{figure}[h]
\definecolor{hleQOrange}{RGB}{250,157,0}
\definecolor{hleQGreen}{RGB}{0,115,59}
\definecolor{hleQBrown}{RGB}{148,102,53}
\definecolor{hleQRed}{RGB}{237,54,36}
\definecolor{hleQMagenta}{RGB}{255,0,212}

\centering
\begin{tikzpicture}[
  x=1.10cm,
  y=1.10cm,
  line cap=round,
  line join=round,
  >=stealth,
  every node/.style={font=\small},
  edgearrow/.style={
    ->,
    line width=.65pt
  },
  increase/.style={
    ->,
    line width=.75pt,
    shorten <=-1pt
  },
  edgelabel/.style={
    fill=white,
    fill opacity=.92,
    text opacity=1,
    inner sep=1pt
  }
]

  \coordinate (v0g) at (0,0);
  \coordinate (v1b) at (1,0);
  \coordinate (v2b) at (2.25,0);
  \coordinate (v0r) at (4.55,0);
  \coordinate (v3)  at (0,1);
  \coordinate (v4)  at (0,2.25);
  \coordinate (v0t) at (0,4.50);
  \coordinate (v2t) at (-2.25,4.50);
  \coordinate (v1t) at (-3.50,4.50);

  \draw[
    black,
    line width=.8pt
  ]
    (-3.85,4.85) -- (1.75,-.75);

  \draw[
    blue!85,
    line width=.8pt
  ]
    (-2.60,4.85) -- (2.95,-.70);

  \draw[
    yellow!90!orange,
    line width=.8pt
  ]
    (0,-.85) -- (0,5.05);

  \draw[
    black,
    line width=.8pt
  ]
    (-.75,0) -- (5.15,0);

  \draw[edgearrow]
    (.23,0) -- (.58,0);

  \draw[edgearrow]
    (1.38,0) -- (1.73,0);

  \draw[edgearrow]
    (3.00,0) -- (3.35,0);

  \draw[
    yellow!90!orange,
    edgearrow
  ]
    (0,.70) -- (0,.35);

  \draw[
    yellow!90!orange,
    edgearrow
  ]
    (0,1.78) -- (0,1.43);

  \draw[
    yellow!90!orange,
    edgearrow
  ]
    (0,3.55) -- (0,3.20);

  \draw[edgearrow]
    (.72,.28) -- (.47,.53);

  \draw[edgearrow]
    (-1.35,2.35) -- (-1.60,2.60);

  \draw[
    blue!85,
    edgearrow
  ]
    (1.18,1.07) -- (.93,1.32);

  \draw[
    blue!85,
    edgearrow
  ]
    (-1.10,3.35) -- (-1.35,3.60);


  \draw[increase]
    (-1.55,2.55) -- (-1.20,2.90);

  \draw[
    blue!85,
    increase
  ]
    (-1.92,4.17) -- (-2.27,3.82);

  \draw[
    yellow!90!orange,
    increase
  ]
    (0,3.72) -- (.42,3.72);

  \draw[increase]
    (3.45,0) -- (3.45,.46);

  \fill[
    hleQRed,
    fill opacity=.42
  ]
    (v1t) circle (9.5pt);

  \fill[
    hleQRed,
    fill opacity=.42
  ]
    (v2t) circle (9.5pt);

  \fill[
    hleQBrown,
    fill opacity=.42
  ]
    (v0t) circle (9.5pt);

  \fill[
    hleQOrange,
    fill opacity=.42
  ]
    (v4) circle (9.5pt);

  \fill[
    hleQOrange,
    fill opacity=.42
  ]
    (v3) circle (9.5pt);

  \fill[
    hleQGreen,
    fill opacity=.42
  ]
    (v0g) circle (9.5pt);

  \fill[
    hleQOrange,
    fill opacity=.42
  ]
    (v1b) circle (9.5pt);

  \fill[
    hleQOrange,
    fill opacity=.42
  ]
    (v2b) circle (9.5pt);

  \fill[
    hleQMagenta,
    fill opacity=.42
  ]
    (v0r) circle (9.5pt);

  \node[
    anchor=north east
  ] at (-3.72,4.25)
    {$v_1$};

  \node[
    anchor=south west
  ] at (-2.12,4.55)
    {$v_2$};

  \node[
    anchor=south west
  ] at (.28,4.42)
    {$v_0$};

  \node[
    anchor=south west
  ] at (.30,2.28)
    {$v_4$};

  \node[
    anchor=north east
  ] at (-.30,1.18)
    {$v_3$};

  \node[
    anchor=north east
  ] at (-.28,-.22)
    {$v_0$};

  \node[
    anchor=north
  ] at (1,-.34)
    {$v_1$};

  \node[
    anchor=north
  ] at (2.25,-.34)
    {$v_2$};

  \node[
    anchor=north
  ] at (4.55,-.34)
    {$v_0$};

  \node[
    edgelabel,
    anchor=north
  ] at (.38,-.28)
    {$E_0$};

  \node[
    edgelabel,
    anchor=north
  ] at (1.58,-.28)
    {$E_3$};

  \node[
    edgelabel,
    anchor=north
  ] at (3.45,-.28)
    {$E_6$};

  \node[
    edgelabel,
    anchor=east
  ] at (-.28,.48)
    {$E_2$};

  \node[
    edgelabel,
    anchor=east
  ] at (-.30,1.58)
    {$E_5$};

  \node[
    edgelabel,
    anchor=west
  ] at (.28,3.70)
    {$E_7$};

  \node[
    edgelabel,
    anchor=west
  ] at (.73,.62)
    {$E_1$};

  \node[
    edgelabel,
    anchor=east
  ] at (-1.62,2.62)
    {$E_8$};

  \node[
    edgelabel,
    anchor=west
  ] at (1.22,1.35)
    {$E_4$};

  \node[
    edgelabel,
    anchor=west
  ] at (-1.28,3.68)
    {$E_9$};

\end{tikzpicture}
\caption{The quotient cellular complex $\faktor{\mathcal{H}_L^\epsilon}{L}$ with monomial labelings from $M_{\Lambda(L)}$}
\label{fig:HLeQuotL}
\end{figure}

\subsection{Cokernel of $\partial^\epsilon_0$}
\label{subsec:coker-example}

Here, we consider the map

\[ (\mathcal{F}^\bullet_{\mathcal{H}_L^\epsilon}, \partial^\epsilon) \stackrel{f}{\rightarrow} S_{\prod_i x_i y_i}\]

given by $\bigoplus_{v\in \mathcal{H}_L^\epsilon}S(-|m_v|) \rightarrow S_{\prod_i x_i y_i}$ where a vertex in $\mathcal{H}_L^\epsilon$ maps to its vertex monomial label in $M_{\Lambda(L)}$.

For $\mathrm{Bl}_p \PP^2$, we have that $L = \left< \left[\begin{matrix} 1 \\ 1 \\ 0 \\ -1\end{matrix}\right] , \left[\begin{matrix} 0 \\ 1 \\ 1 \\ -1 \end{matrix}\right]\right>$, \\

\begin{align*}
I_L &= \left< x^{v_+}-x^{v_-} \text{ }|\text{ } v\in L \right> \\
 &= \left< x_1x_2-x_4, x_2x_3-x_4, x_1-x_3\right>, \end{align*}

 \begin{align*}
J_L &= \left< x^uy^v - x^vy^u \text{ }|\text{ }u-v\in L\right>\\
&= \left<x_1x_2y_4 - x_4y_1y_2, x_2x_3y_4 - x_4y_2y_3, x_1y_3-x_3y_1\right> \end{align*}

and $M_{\Lambda(L)}$ as an $S-$submodule of $T$ has the infinite generating set $\{ x^uy^{-u}\text{ }|\text{ }u\in L \}$. In the quotient by $L$, the cellular complex $\mathcal{H}_L^\epsilon$ carries monomial labels given in Figure~\ref{fig: HLL.12.7}. When we consider monomial labels on vertices in the quotient $\mathcal{H}_L^\epsilon/L$, we only need to consider the monomial labels on vertices given in Figure~\ref{fig: HLL.12.7} due to which vertices in $\mathcal{H}_L^\epsilon(0)$ carry the same monomial label. \\

Here, $L\cap \N^n = \{0\}$ implies that $M_{\Lambda(L)}$ contains $1$, so that $M_{\Lambda(L)} \subset Im(f)$ above, by considering the action of $L$ on the quotient from Figure~\ref{fig: HLL.12.7}. Modulo the action of $L$, there is one additional monomial $\frac{y_2}{x_2}$ in $Im(f)$ which is not contained in $M_{\Lambda(L)}$. Note that since 

\[ \{ \frac{x_1x_2 y_4}{x_4y_1y_2}, \frac{x_2x_3y_4}{x_4y_2y_3}, \frac{x_1y_3}{x_3y_1} \} \subset M_{\Lambda(L)},\]

we also have

\[ \{\frac{x_4y_1y_2}{x_1x_2y_4}, \frac{x_4y_2y_3}{x_2x_3y_4}, \frac{x_3y_1}{x_1y_3} \} \subset M_{\Lambda(L)}. \] 

Now for $X_\Sigma = \mathrm{Bl}_p\PP^2$,

\begin{align*}
I_{irr} &= \left< x_3x_4, x_1x_4, x_1x_2, x_2x_3\right> \\
&= \left<x_1, x_3\right> \cap \left<x_2,x_4\right> \end{align*}

and $I_{irr}$ for $X_\Sigma \times X_\Sigma$ is given by \[ \left<x_1, x_3\right>\cap\left<x_2,x_4\right> \cap \left<y_1,y_3\right> \cap \left<y_2, y_4\right>. \] 

Now, for any monomial $q$ in $I_{irr}$ for $X_\Sigma \times X_\Sigma$, we must have that either $x_2$ or $x_4$ divides $q$. If $x_2$ divides $q$, then

\[ x_2 \left( \frac{y_2}{x_2} \right) = y_2 \in M_{\Lambda(L)} \] 

since $M_{\Lambda(L)}$ contains $1$ shows that $q \cdot \frac{y_2}{x_2} \in M_{\Lambda(L)}$. If $x_4 | q$, then 

\begin{align*}
y_3x_4 \cdot \left( \frac{y_2}{x_2}\right) &= \left(\frac{x_4y_2y_3}{x_2x_3y_4}\right) \cdot x_3y_4 \in M_{\Lambda(L)} \end{align*}

shows that $q \cdot \left(\frac{y_2}{x_2}\right)$ is in the image of the action of $S$ on $q\cdot \left( \frac{y_2}{x_2}\right)$. Therefore, the cokernel of the inclusion of $M_{\Lambda(L)} \hookrightarrow Im(f)$ is torsion with respect to $I_{irr}$ for $X_\Sigma \times X_\Sigma$. While an $R-$module $M$ is $I$-torsion if and only if there exists $k\in \Z_+$ such that $I^k M = 0$, here $k=1$ suffices to show that the cokernel of the inclusion of $M_{\Lambda(L)} \hookrightarrow Im(f)$ is torsion with respect to $I_{irr}$ for $X_\Sigma \times X_\Sigma$.

\FloatBarrier
\begin{figure}[h]
\begin{tikzpicture}[
  x=1.12cm,
  y=1.12cm,
  line cap=round,
  line join=round,
  >=stealth,
  every node/.style={font=\small},
  hyperplane/.style={
    line width=.7pt
  },
  increase/.style={
    ->,
    line width=.85pt,
    shorten <=-1.1pt
  }
]

  \draw[
    hyperplane,
    black
  ]
    (-3.55,0) -- (5.25,0);

  \draw[
    hyperplane,
    yellow!85!orange
  ]
    (0,-2.10) -- (0,3.95);

  \draw[
    hyperplane,
    black
  ]
    (-2.60,3.60) -- (3.80,-2.80);

  \draw[
    hyperplane,
    blue!85
  ]
    (-1.85,3.85) -- (4.05,-2.05);


  \draw[
    increase,
    yellow!85!orange
  ]
    (0,3.62) -- (.48,3.62);

  \draw[
    increase,
    black
  ]
    (4.15,0) -- (4.15,.58);

  \draw[
    increase,
    black
  ]
    (2.05,-1.05) -- (2.42,-.68);

  \draw[
    increase,
    blue!85
  ]
    (3.55,-1.55) -- (3.18,-1.92);

  \fill[
    orange!70,
    fill opacity=.48
  ]
    (0,2) circle (10pt);

  \node[
    anchor=west
  ] at (.20,2.20)
    {$1$};

  \fill[
    orange!70,
    fill opacity=.48
  ]
    (0,1) circle (10pt);

  \node[
    anchor=east
  ] at (-.20,.78)
    {$1$};

  \fill[
    green!55!black,
    fill opacity=.38
  ]
    (0,0) circle (10pt);

  \node[
    anchor=north east
  ] at (-.25,-.20)
    {$\displaystyle\frac{y_2}{x_2}$};

  \fill[
    orange!70,
    fill opacity=.48
  ]
    (1,0) circle (10pt);

  \node[
    anchor=north
  ] at (1,-.25)
    {$1$};

  \fill[
    orange!70,
    fill opacity=.48
  ]
    (2,0) circle (10pt);

  \node[
    anchor=south
  ] at (2,.25)
    {$1$};

  \node[
    anchor=west
  ] at (.12,3.30)
    {$x_1=0$};

  \node[
    anchor=south west
  ] at (4.38,.10)
    {$x_3=0$};

  \node[
    anchor=north east
  ] at (1.72,-1.28)
    {$x_2=0$};

  \node[
    anchor=north west
  ] at (2.35,-.60)
    {$x_4=0$};

\end{tikzpicture}
\caption{Monomial labels on vertices $v\in \mathcal{H}_L^\epsilon(0)$ modulo $L$}
\label{fig: HLL.12.7}
\end{figure}
\FloatBarrier

\subsection{Changing the secondary-fan chamber}

Here we consider $\text{coker}(M_{\Lambda(L)}\hookrightarrow\operatorname{Im}(f))$ in the other chamber of the effective cone, spanned by $\mathcal{O}(D_2)$ and $\mathcal{O}(D_4)$. Fix $0<t<\frac12$, take the class $\epsilon=t([D_2]+[D_4])$ in the relative interior of this chamber, and choose the lift
\[
a=t(0,1,0,1)\in\R^4.
\]
For the circuits $c_1,c_2,c_3$ listed above, the pairings with this lift are $-t,2t,$ and $t$, respectively, so $a$ is circuit-generic. Thus the chamber condition selects the new GIT phase, while the separate circuit condition guarantees transversality. (Again, $D_2$ corresponds to the exceptional divisor of the blow-up of $\mathrm{Bl}_p\PP^2$.) Our previous discussion of $\text{coker}(M_{\Lambda(L)}\hookrightarrow\operatorname{Im}(f))$ from Section~\ref{subsec:coker-example} took place in the nef chamber of $N^1(X_\Sigma)$, with $N^1(X_\Sigma)\cong\mathrm{Cl}(X_\Sigma)\otimes_\Z\R$ since $X_\Sigma$ is smooth. We now study the same cokernel in the other chamber shown in Figure~\ref{fig:effectivecone}.

\FloatBarrier
\begin{figure}[h]
\definecolor{effectiveGreen}{RGB}{29,177,71}
\centering
\begin{tikzpicture}[
  x=1cm,
  y=1cm,
  line cap=round,
  line join=round,
  >=stealth,
  every node/.style={font=\large}
]

  \coordinate (O) at (0,0);

  %
  \begin{scope}
    \clip (0,0) rectangle (5.95,5.72);

    \fill[
      effectiveGreen,
      fill opacity=.90
    ]
      (O) -- (20,0) -- (20,19.25) -- cycle;
  \end{scope}

  \draw[
    ->,
    line width=.85pt
  ]
    (O) -- (7.55,0);

  \draw[
    ->,
    line width=.85pt
  ]
    (O) -- (0,6.38);

  \draw[
    ->,
    line width=.85pt
  ]
    (O) -- (6.28,6.04);

  \fill (O) circle (2.8pt);

  \node[
    anchor=south east
  ] at (-.10,6.35)
    {$\mathcal{O}(D_1),\mathcal{O}(D_3)$};

  \node[
    anchor=south west
  ] at (6.06,6.18)
    {$\mathcal{O}(D_4)$};

  \node[
    anchor=west
  ] at (7.62,0)
    {$\mathcal{O}(D_2)$};

  \begin{scope}[
    shift={(2.55,4.65)},
    line width=.48pt
  ]
    \draw (0,0) -- (0,.86);
    \draw (0,0) -- (.92,0);
    \draw (0,0) -- (.78,.78);
    \draw (0,0) -- (-.78,-.78);
  \end{scope}

  \begin{scope}[
    shift={(4.55,2.02)}
  ]
    \fill[white]
      (-.73,-.67) rectangle (.84,.80);

    \begin{scope}[
      line width=.48pt
    ]
      \draw (0,0) -- (0,.68);
      \draw (0,0) -- (.70,0);
      \draw (0,0) -- (-.64,-.64);
    \end{scope}
  \end{scope}

\end{tikzpicture}
\caption{Chambers of the Effective Cone}
\label{fig:effectivecone}
\end{figure}
\FloatBarrier

Inside of this cone spanned by $\mathcal{O}(D_2)$ and $\mathcal{O}(D_4)$ in $\mathrm{Cl}(X_\Sigma)\otimes_\Z \R$, we have a corresponding fan for which the ray $\rho_2$ is removed. This gives the fan for $\PP^2$, though we retain the information that $\rho_2$ no longer lives in any maximal cone so that the irrelevant ideal for $X_\Sigma$ becomes

\[ I_{irr} = \left<x_1,x_3,x_4\right> \cap (x_2) \]

and the irrelevant ideal for $X_\Sigma \times X_\Sigma$ in this case is

\[ I_{irr} = \left<x_1,x_3,x_4\right> \cap (x_2)\cap \left<y_1, y_3, y_4\right> \cap (y_2). \] 

For the lift $a=t(0,1,0,1)$, we obtain the deformed cellular complex $\mathcal{H}_L^\epsilon$, whose quotient by $L$ is given in Figure~\ref{fig: def2}. Here, the second and fourth coordinate hyperplane families are shifted; they are labeled in blue and yellow, respectively, in the figure. 

\FloatBarrier
\begin{figure}[h]
\centering
\begin{tikzpicture}[
  x=1.15cm,
  y=1.15cm,
  line cap=round,
  line join=round,
  >=stealth,
  every node/.style={font=\small},
  hyperplane/.style={
    line width=.8pt
  },
  increase/.style={
    ->,
    line width=.95pt,
    shorten <=-1.1pt
  },
  hplabel/.style={
    fill=white,
    fill opacity=.88,
    text opacity=1,
    inner sep=1.5pt
  }
]

  \draw[
    hyperplane,
    black
  ]
    (0,-3.05) -- (0,3.25)
    node[
      hplabel,
      pos=.91,
      right
    ] {$x_1=0$};

  \draw[
    hyperplane,
    black
  ]
    (-3.85,0) -- (4.85,0)
    node[
      hplabel,
      pos=.92,
      above
    ] {$x_3=0$};

  \draw[
    hyperplane,
    yellow!90!black
  ]
    (-2.55,3.55) -- (4.25,-3.25)
    node[
      hplabel,
      text=yellow!50!black,
      pos=.84,
      sloped,
      above
    ] {$x_2=0$};

  \draw[
    hyperplane,
    blue!85
  ]
    (-3.75,2.75) -- (2.15,-3.15)
    node[
      hplabel,
      text=blue!85,
      pos=.17,
      sloped,
      below
    ] {$x_4=0$};


  \draw[
    increase,
    black
  ]
    (0,2.05) -- (.72,2.05);

  \draw[
    increase,
    black
  ]
    (2.35,0) -- (2.35,.82);

  \draw[
    increase,
    yellow!90!black
  ]
    (-1.78,2.78) -- (-2.18,2.38);

  \draw[
    increase,
    blue!85
  ]
    (1.02,-2.02) -- (1.42,-1.62);

  \fill[
    orange!70,
    fill opacity=.48
  ]
    (0,1) circle (10pt);

  \node[
    anchor=west
  ] at (.22,1.22)
    {$1$};

  \fill[
    orange!70,
    fill opacity=.48
  ]
    (0,0) circle (10pt);

  \node[
    anchor=north west
  ] at (.18,-.08)
    {$1$};

  \fill[
    orange!70,
    fill opacity=.48
  ]
    (1,0) circle (10pt);

  \node[
    anchor=north west
  ] at (1.18,.10)
    {$1$};

  \fill[
    green!55!black,
    fill opacity=.38
  ]
    (-1,0) circle (10pt);

  \node[
    anchor=north east
  ] at (-1.10,-.12)
    {$\displaystyle\frac{y_1}{x_1}$};

  \fill[
    green!55!black,
    fill opacity=.38
  ]
    (0,-1) circle (10pt);

  \node[
    anchor=south west
  ] at (-.52,-1.85)
    {$\displaystyle\frac{y_3}{x_3}$};

\end{tikzpicture}
\caption{Deformed complex $\mathcal{H}_L^\epsilon$ for the class $\epsilon=t([D_2]+[D_4])$}
\label{fig: def2}
\end{figure}
\FloatBarrier

In this case, $I_{irr}$ for $X_\Sigma \times X_\Sigma$ is given above and $\text{coker}(M_{\Lambda(L)}) \ni \frac{y_1}{x_1}, \frac{y_3}{x_3}$ and is spanned by translates of these elements by $L$. Note that in this case, at least one of $x_1, x_3$, or $x_4$ must divide any monomial $q$ in the irrelevant ideal $I_{irr}$ of $X_\Sigma\times X_\Sigma$. \\[.2cm]

For $L$-translates of $\frac{y_1}{x_1}$, we note that if $x_1$ divides $q$, then 

\begin{align*}
x_1 \cdot \left(\frac{y_1}{x_1} \right)=y_1 \in M_{\Lambda(L)}\end{align*} since $1\in M_{\Lambda(L)}$, and $M_{\Lambda(L)}$ is an $S-$submodule of $T= k[x_1^{\frac{+}{}}, \dots, x_n^{\frac{+}{}}, y_1^{\frac{+}{}}, \dots, y_n^{\frac{+}{}}]$. If $x_3$ divides $q$, then 

\begin{align*}
x_3 \cdot \frac{y_1}{x_1} &= y_3 \cdot \left(\frac{x_3y_1}{x_1y_3}\right) \in M_{\Lambda(L)}. \end{align*}

If $x_4$ divides $q$, then

\begin{align*}
x_4y_2 \left(\frac{y_1}{x_1}\right) &= x_2y_4\left(\frac{x_4y_1y_2}{x_1x_2y_4}\right) \in M_{\Lambda(L)}\end{align*}

so that $\overline{\frac{y_1}{x_1}} = \overline{0} \in \text{coker}(M_{\Lambda(L)}\hookrightarrow Im(f) )$.

Next, for ($L-$translates of) $\frac{y_3}{x_3}$, we again have that at least one of $x_1, x_3$, or $x_4$ divides any monomial $q$ in $I_{irr}$, the irrelevant ideal for $X_\Sigma\times X_\Sigma$. If $x_1$ divides $q$, then 

\begin{align*}
\frac{y_3}{x_3}\cdot x_1 &= \left(\frac{x_1y_3}{x_3y_1}\right)\cdot y_1 \in M_{\Lambda(L)}\end{align*}

If $x_3$ divides $q$, then 

\begin{align*} \frac{y_3}{x_3}\cdot x_3 &= y_3\in M_{\Lambda(L)} \end{align*} since $1 \in M_{\Lambda(L)}$. Lastly, if $x_4 | q$, then \begin{align*} \left(\frac{y_3}{x_3}\right)x_4y_2 &= \left(\frac{x_4y_2y_3}{x_2x_3y_4}\right)x_2y_4 \in M_{\Lambda(L)}. \end{align*}

This shows that $\overline{\frac{y_3}{x_3}} = 0 \in \text{coker}(M_{\Lambda(L)}\hookrightarrow Im(f))$. Hence,

\[ \text{coker}(M_{\Lambda(L)}\hookrightarrow Im(f)) = 0 \] 

modulo $I_{irr}$ for the GIT phase determined by the class $\epsilon=t([D_2]+[D_4])$, with the circuit-generic lift $a=t(0,1,0,1)$ and $0<t<\frac12$. 

\section{Torsion of the remaining homology and proof of Theorem~\ref{thm: MainTheorem1}}
\label{sec:torsion}

\begin{definition}\label{def:centralsym}
A complete fan $\Sigma\subset N_\R$ is \textbf{centrally symmetric} if $\Sigma(1)$ is invariant under $u\mapsto -u$ and its maximal cones come in opposite pairs.
\end{definition}

\begin{remark}\label{rem:counterexample}
The undeformed construction can fail without additional hypotheses. A counterexample can be constructed from a smooth complete toric surface with rays
\[
(1,0),(0,1),(-1,0),(-2,-1),(-1,-1),(-1,-2),(0,-1)
\]
for which the degree-zero homology of the cellular complex has an associated prime $\langle y_0,y_1\rangle$, where $y_i$ denotes the $y$-variable corresponding to the $i$-th ray in the displayed ordering; in particular, $y_0$ and $y_1$ correspond to the rays $(1,0)$ and $(0,1)$, respectively. Hence the complex does not resolve the diagonal.
\end{remark}

Proposition~\ref{prop: main} shows that the cokernel of the augmentation map in homological degree~$0$ agrees with the diagonal after saturating by the irrelevant ideal. To conclude Theorem~\ref{thm: MainTheorem1}, it remains to show that the 
\emph{other} homology of the cellular complex is also killed by a power of the irrelevant ideal. We give a criterion for this in terms of the labeled subcomplexes $X_{\leq \mathbf{b}}$ and then verify it under the central symmetry hypothesis of Definition~\ref{def:centralsym}.

\subsection{A torsion criterion via labeled subcomplexes}
Let $\widetilde X:=\mathcal{H}_L$ (for $\epsilon=0$) and $X:=\widetilde X/L$. The labeled cell complex $\widetilde X$ determines an $L$-equivariant $\Z^{2n}$-graded $S[\Lambda(L)]$-free cellular complex $(\widetilde{\mathcal{F}}^\bullet_{\widetilde X},\partial)$, whose image under the Bayer--Sturmfels equivalence \cite[Theorem~3.2]{bayer-sturmfels} is the finite $\Z^{2n}/\Lambda(L)$-graded $S$-free complex $(\mathcal{F}^\bullet_X,\partial)$ used throughout.

For a Laurent monomial $\mathbf{b}\in S_{\prod x_i y_i}$ we write $\widetilde X_{\leq \mathbf{b}}$ for the subcomplex of $\widetilde X$ consisting of those faces $F$ with label $m_F$ dividing $\mathbf{b}$ (equivalently, with multidegree $\leq \deg(\mathbf{b})$ in the coordinatewise partial order).

The cellular construction identifies multigraded homology with reduced homology of these lifted subcomplexes (see, e.g., \cite[Chapter~4]{miller-sturmfels} and \cite[Proposition~1.1]{bayer-sturmfels}): for each multidegree $\mathbf{b}$ we have
\[
H^i(\mathcal{F}^\bullet_X)_\mathbf{b}\ \cong\ \widetilde{H}_{i-1}(\widetilde X_{\leq \mathbf{b}};\Bbbk),\qquad i\ge 1,
\]
and similarly the kernel in homological degree~$0$ is controlled by $\widetilde{H}_0(\widetilde X_{\leq \mathbf{b}};\Bbbk)$.
Consequently, to show that $H^i(\mathcal{F}^\bullet_X)$ is $I_{irr}$-torsion it suffices to show that for every multidegree $\mathbf{b}$ there exists $k\gg 0$ and a monomial $m\in I_{irr}$ such that the inclusion-induced map
\[
\widetilde{H}_*(\widetilde X_{\leq \mathbf{b}};\Bbbk)\longrightarrow \widetilde{H}_*(\widetilde X_{\leq m^k\mathbf{b}};\Bbbk)
\]
is the zero map.

\subsection{A symmetric irrelevant monomial}
Assume now that $\Sigma$ is centrally symmetric in the sense of Definition~\ref{def:centralsym}, and fix a maximal cone $\sigma\in\Sigma(n)$. Since $-\sigma\in\Sigma(n)$ as well, the corresponding Cox monomials $x^{\hat\sigma}$ and $x^{\widehat{-\sigma}}$ are generators of the irrelevant ideal on $X_\Sigma$. In particular, the product $x^{\hat\sigma}x^{\widehat{-\sigma}}$ is divisible by $\prod_{\rho\in\Sigma(1)} x_\rho$ because every ray of $\Sigma(1)$ is omitted from at least one of $\sigma(1)$ and $(-\sigma)(1)$. The same holds for the $y$-variables on the second factor. Hence the monomial
\[
m_{\mathrm{sym}}\ :=\ x^{\hat\sigma}x^{\widehat{-\sigma}}\,y^{\hat\sigma}y^{\widehat{-\sigma}}
\]
lies in $I_{irr}^2\subset S$ and has positive exponent in every variable $x_\rho$ and $y_\rho$.

\subsection{Proof of Theorem~\ref{thm: MainTheorem1}}
Let $\mathbf{b}$ be any multidegree. Since $m_{\mathrm{sym}}$ has positive exponent in every variable, for $k\gg 0$ the divisibility condition $m_F\mid m_{\mathrm{sym}}^k\mathbf{b}$ imposes both upper and lower bounds on the coordinates of vertices in $\widetilde X_{\le m_{\mathrm{sym}}^k\mathbf{b}}\subset \R L$. Concretely, the set of contributing vertices is contained in an axis-parallel box in $\R^{\Sigma(1)}$; intersecting with the affine subspace $\R L$ produces a bounded convex polytope. The restriction of the hyperplane arrangement to this polytope gives a regular cell decomposition, hence $\widetilde X_{\le m_{\mathrm{sym}}^k\mathbf{b}}$ is contractible.

It follows that $\widetilde{H}_{i-1}(\widetilde X_{\le m_{\mathrm{sym}}^k\mathbf{b}};\Bbbk)=0$ for all $i\ge 1$, i.e.
\[
H^i(\mathcal{F}^\bullet_X)_{m_{\mathrm{sym}}^k\mathbf{b}}=0\qquad (i>0).
\]
Therefore multiplication by $m_{\mathrm{sym}}^k$ annihilates each multigraded piece of $H^i(\mathcal{F}^\bullet_X)$ for $i>0$. Since $m_{\mathrm{sym}}\in I_{irr}^2$, we conclude that $H^i(\mathcal{F}^\bullet_X)$ is $I_{irr}$-torsion for every $i>0$.

The same contractibility argument shows $\widetilde{H}_0(\widetilde X_{\le m_{\mathrm{sym}}^k\mathbf{b}};\Bbbk)=0$, hence the kernel of the augmentation in homological degree~$0$ is also $I_{irr}$-torsion. Combining this with Proposition~\ref{prop: main} (which controls the cokernel after saturation) shows that after sheafification the complex is exact with cokernel $\mathcal{O}_\Delta$. This proves Theorem~\ref{thm: MainTheorem1}.

\section{Smooth, non-unimodular example: Twice-Iterated Blow-up of $\PP^2$ at a point}
\justifying
Let $X_\Sigma$ be the projective toric surface obtained from $\PP^2$ by two successive torus-equivariant blow-ups, with the second center a torus-fixed point on the exceptional curve created by the first. The resulting surface is smooth, while its principal-divisor lattice is non-unimodular. Its fan $\Sigma$ is given by $\Sigma(1) = \{ e_1, e_1 + e_2, e_2, -e_1 + e_2, -e_2 \} \subset N$ which we denote by $\rho_1, \dots, \rho_5$ and \newline $\Sigma(2)= \left\{ \{ \rho_1, \rho_2\}, \{\rho_2, \rho_3\}, \{\rho_3, \rho_4\}, \{\rho_4, \rho_5\}, \{\rho_5, \rho_1\} \right\}$. 
In this section we illustrate the construction at $\epsilon=0$, i.e. $\mathcal{H}_L^\epsilon=\mathcal{H}_L$ without deformation. Proposition~\ref{prop: main} ensures that the cokernel agrees with the diagonal after saturating by the irrelevant ideal; without additional hypotheses we do not claim that the complex is exact on the nose (cf. Remark~\ref{rem:counterexample}). Here, $N \cong \Z^2$ and $X_\Sigma$ has fundamental exact sequence

\[ 0 \rightarrow \Z^2 \stackrel{B}{\rightarrow} \Z^5 \stackrel{\pi}{\rightarrow} \text{Cl}(X) \rightarrow 0 \] 

with $L = \text{Im}(B)$ a rank $2$ lattice in $\Z^5$. Here $B = \left[ \begin{matrix} 1 & 0 \\ 1 & 1 \\ 0 & 1 \\ -1 & 1 \\ 0 & -1 \end{matrix}\right]$ with columns $\vec{u}$ and $\vec{v}$, respectively. Here $\mathrm{Cl}(X_\Sigma)$ has presentation
\begin{align*}
    \mathrm{Cl}(X_\Sigma) & \cong \faktor{ \left< D_1, \dots, D_5 \right> }{(D_4 \sim D_1 + D_2, D_5 \sim D_1 + 2D_2 + D_3)}
\end{align*}
so we use the basis $\{D_1, D_2, D_3 \}$ for $\mathrm{Cl}(X_\Sigma)$. The infinite hyperplane arrangement $\mathcal{H}_L$ intersected with $\mathbb{R}L$ is given in Figure~\ref{fig: bp2HL}. A fundamental domain for $\faktor{(\mathcal{H}_L \cap \R L)}{L}$ is given in Figure~\ref{fig: bp2funddom}. Monomial labelings of faces are given in Table~\ref{table: monbp2}. This gives the complex $(\mathcal{F}^\bullet_{\mathcal{H}_L/L}, \partial)$ of free $S$-modules graded by $\mathrm{Cl}(X)$, and hence a complex of line bundles on $X_\Sigma\times X_\Sigma$. Write
\[
\begin{aligned}
\mathcal{E}_2={}&
\mathcal{O}((-1,-1,0)\times(-2,-3,-1))\\
&\oplus\mathcal{O}((-1,-2,-1)\times(-2,-2,-1))\\
&\oplus\mathcal{O}((-1,-1,0)\times(-1,-1,-1))\\
&\oplus\mathcal{O}((-1,-2,-1)\times(-1,-2,-1)),\\[0.4em]
\mathcal{E}_1={}&
\mathcal{O}((-1,-1,0)\times(-1,-1,0))\\
&\oplus\mathcal{O}((-1,-1,0)\times(-2,-2,-1))\\
&\oplus\mathcal{O}((0,-1,0)\times(-1,-2,-1))\\
&\oplus\mathcal{O}((-1,-2,-1)\times(-1,-2,-1))\\
&\oplus\mathcal{O}((0,0,0)\times(-1,-1,-1))\\
&\oplus\mathcal{O}((0,0,0)\times(-1,-1,-1)),\\[0.4em]
\mathcal{E}_0={}&
\mathcal{O}((0,0,0)\times(0,0,0))
\oplus
\mathcal{O}((1,1,1)\times(-1,-1,-1)).
\end{aligned}
\]
The complex is
\[
0\longrightarrow\mathcal{E}_2
\xrightarrow{\partial_2=d^{-2}}\mathcal{E}_1
\xrightarrow{\partial_1=d^{-1}}\mathcal{E}_0
\longrightarrow0.
\]

The boundary maps are given by

\begin{align*}
    \partial_2 &= \left( \begin{matrix} 0 & 0 & x_1x_2 & -x_5y_2 \\ 0 & x_5y_4 & -x_4 & 0 \\ 0 & y_1 & 0 & -1 \\ x_1y_4 & 0 & 0 & x_3x_4 \\ y_2 & x_2x_3 & 0 & 0 \\ y_5 & 0 & y_3 & 0 \end{matrix}\right) 
\end{align*}
and
\begin{align*}
    \partial_1 &= \left( \begin{matrix} x_4 y_1 y_2 - x_1 x_2 y_4 & -x_4 y_1 y_5 & x_2 y_5 & x_2 x_3 x_4 y_5 - x_5 y_2 y_3 y_4 & y_1 y_2 y_3 & y_3 y_4 \\ 0 & x_1 x_5 y_4 & -x_5 y_2 & 0 & -x_1 x_2 x_3 & -x_3 x_4 \end{matrix}\right)
\end{align*}
subject to the ordering of faces given in Figure~\ref{fig: bp2funddom}. In Figure~\ref{fig: bp2funddom}, there is an extra vertex because $X_\Sigma$ is not unimodular. This vertex did not appear in \cite{bayer-popescu-sturmfels}. This extra vertex is why we use the floor function in our monomial labeling on vertices. We remind the reader that the grading on monomial labelings here is given as an element of $\mathrm{Cl}(X_\Sigma)$. \\


\subsection{Implications for exceptional collections}

As mentioned previously, Beilinson \cite{Beilinson1978} gave a locally-free resolution of the structure sheaf of the diagonal of length $n$ for $\Delta \subset \PP^n \times \PP^n$. Considering just the line bundles which appear on the left-hand side yields the full strong exceptional collection \cite{huybrechts2006fourier} \[ \mathcal{E} = \{ \mathcal{O}_{\PP^n}(-n), \mathcal{O}_{\PP^n}(-n+1), \dots, \mathcal{O}_{\PP^n} \}.\] 

King conjectured in unpublished notes that any smooth projective toric variety admits a full strong exceptional collection of line bundles. Counterexamples to King's conjecture were given in \cite{hille2006counterexamplekingsconjecture} and \cite{Efimov2014}, so that whether a smooth projective toric variety admits a full strong exceptional collection of line bundles is now checked on a case-by-case basis. The author investigates implications of resolutions of the diagonal to full strong exceptional collections of line bundles in \cite{anderson2024exceptionalcollectionslinebundles} and \cite{ramirez2025exceptionalcollectionstoricfano} for smooth projective toric varieties in dimensions 2-4. Exceptional collections of line bundles for smooth projective toric Fano varieties in dimension 4 were also constructed in \cite{prabhunaik2015tiltingbundlestoricfano}. When a full strong exceptional collection of line bundles on a smooth projective toric variety exists, it must have length equal to the rank of the K-theory of $X_\Sigma$, which is also the topological euler characteristic of $X_\Sigma$. Below we discuss implications of our resolution of the diagonal for full strong exceptional collections of line bundles for the twice-iterated blow-up of $\PP^2$ at a point.

\subsection{Exceptional collections for twice-iterated blow-up of $\PP^2$ at a point}

Considering just the line bundles which appear on the left-hand side gives the collection \[ \mathcal{E} = \{\mathcal{O}(-1,-2,-1), \mathcal{O}(-1,-1,0), \mathcal{O}(0,-1,0), \mathcal{O}, \mathcal{O}(1,1,1) \} \] 

with

\begin{align*}
    \mathrm{Ext}^\bullet_{\mathcal{O}_{X_\Sigma}}(\mathcal{O}(D), \mathcal{O}(E)) \cong H^*(X_\Sigma, \mathcal{O}(E-D)) = H^*(X_\Sigma, \mathcal{O}_{X_\Sigma})
\end{align*}

for $D=E$ gives $$\mathrm{Ext}^\bullet_{\mathcal{O}_{X_\Sigma}}(\mathcal{O}(D), \mathcal{O}(D)) = \begin{cases} \Bbbk & i=0 \\ 0 & \text{ else }\end{cases} $$ for all $\mathcal{O}(D)$ appearing in $\mathcal{E}$. Direct computation yields that $\mathrm{Hom}(E_i, E_j[d]) = 0$ for all $i>j, d\in \Z$, so $\mathcal{E}$ gives an exceptional collection of line bundles on $X_\Sigma$ \cite[Definition 1.57]{huybrechts2006fourier}. We also note that the rank of the K-theory of $X_\Sigma$ is $5$, which matches the number of objects appearing in $\mathcal{E}$. $\mathrm{Ext}^\bullet_{\mathcal{O}_{X_\Sigma}}(E_j, E_i)$ for $i>j$ (i.e., Hom's `to the right') are given in Figure~\ref{fig:exchoms}. We calculate $\mathrm{Hom}^{\bullet}(E_j, E_i[d])$ for $i>j, d\in \Z$ by, respectively, computing $H^*(X_\Sigma, \mathcal{O}_{X_\Sigma}(a_1,a_2,a_3))$ for $(a_1, a_2, a_3)$ in the following set $\mathcal{S}$:
\begin{align*}
 \mathcal{S} &=  \{(0,-1,-1),(-1,-1,-1),(-1,-2,-1),(-2,-3,-2),(-1,0,0),(-1,-1,0),(-2,-2,-1),(0,-1,0)\} 
 \end{align*}
One can show that this is given by the reduced cohomology of $V_{D,m}$ using the methods of \cite[Section 9.1 and Proposition 9.1.6]{C-L-S}.

\FloatBarrier

\begin{figure}[h]
\definecolor{bpBlue}{RGB}{0,111,255}
\definecolor{bpRed}{RGB}{237,54,36}
\definecolor{bpGreen}{RGB}{0,115,59}
\definecolor{bpPurple}{RGB}{164,0,199}
\definecolor{bpTeal}{RGB}{25,176,146}

\centering
\begin{tikzpicture}[
  x=.88cm,
  y=.88cm,
  line cap=round,
  line join=round,
  >=stealth,
  every node/.style={font=\small},
  hyperplane/.style={
    line width=.65pt
  },
  increase/.style={
    ->,
    line width=.8pt,
    shorten <=-1pt
  },
  hplabel/.style={
    fill=white,
    fill opacity=.88,
    text opacity=1,
    inner sep=1.2pt
  }
]

  \def\xmin{-4.15}
  \def\xmax{5.45}
  \def\ymin{-4.15}
  \def\ymax{4.35}

  \begin{scope}
    \clip (\xmin,\ymin) rectangle (\xmax,\ymax);

    \foreach \k in {1,2,3,4}{
      \draw[
        hyperplane,
        bpBlue
      ]
        (\k,\ymin) -- (\k,\ymax);
    }

    \foreach \k in {0,1,2,3}{
      \draw[
        hyperplane,
        bpRed
      ]
        (\xmin,{\k-\xmin})
        --
        (\xmax,{\k-\xmax});
    }

    \foreach \k in {0,1,2,3}{
      \draw[
        hyperplane,
        bpPurple
      ]
        (\xmin,{\xmin+\k})
        --
        (\xmax,{\xmax+\k});
    }

    \draw[
      hyperplane,
      bpGreen
    ]
      (\xmin,3) -- (\xmax,3);

    \draw[
      hyperplane,
      bpTeal
    ]
      (\xmin,2) -- (\xmax,2);

    \draw[
      hyperplane,
      bpGreen
    ]
      (\xmin,1) -- (\xmax,1);

    \draw[
      hyperplane,
      bpTeal
    ]
      (\xmin,-1) -- (\xmax,-1);
  \end{scope}

  \draw[
    ->,
    black,
    line width=.72pt
  ]
    (\xmin,0) -- (\xmax+.18,0)
    node[
      anchor=west
    ]
      {$+\vec u$};

  \draw[
    ->,
    black,
    line width=.72pt
  ]
    (0,\ymin) -- (0,\ymax+.18)
    node[
      anchor=south
    ]
      {$+\vec v$};


  \node[
    hplabel,
    text=bpBlue,
    anchor=south
  ] at (1,4.16)
    {$x_1$};

  \draw[
    bpBlue,
    increase
  ]
    (1,3.72) -- (1.58,3.72);

  \node[
    text=bpRed,
    anchor=south east
  ] at (-3.40,3.98)
    {$x_2$};

  \draw[
    bpRed,
    increase
  ]
    (-3.12,3.12) -- (-2.70,3.54);

  \node[
    text=bpPurple,
    anchor=south west
  ] at (3.24,4.04)
    {$x_4$};

  \draw[
    bpPurple,
    increase
  ]
    (2.92,3.92) -- (2.50,4.34);

  \node[
    hplabel,
    text=bpGreen,
    anchor=east
  ] at (5.12,1)
    {$x_3$};

  \draw[
    bpGreen,
    increase
  ]
    (5.20,1) -- (5.20,1.48);

  \node[
    hplabel,
    text=bpTeal,
    anchor=west
  ] at (4.98,-1)
    {$x_5$};

  \draw[
    bpTeal,
    increase
  ]
    (4.75,-1) -- (4.75,-1.48);

\end{tikzpicture}
\caption{$\mathcal{H}_L \cap \R L$ for $X_\Sigma$}
\label{fig: bp2HL}
\end{figure}

\FloatBarrier
\begin{figure}[h]
\definecolor{bpfdTop}{RGB}{0,174,145}
\definecolor{bpfdBottom}{RGB}{0,128,75}
\definecolor{bpfdBlue}{RGB}{0,111,255}
\definecolor{bpfdRed}{RGB}{244,55,43}
\definecolor{bpfdPurple}{RGB}{175,0,210}

\centering
\begin{tikzpicture}[
  x=1.02cm,
  y=1.02cm,
  line cap=round,
  line join=round,
  >=stealth,
  every node/.style={font=\small}
]

  \coordinate (NW) at (0,6);
  \coordinate (NE) at (6,6);
  \coordinate (SW) at (0,0);
  \coordinate (SE) at (6,0);
  \coordinate (C)  at (3,3);


  \draw[
    bpfdTop,
    line width=.75pt
  ]
    (NW) -- (NE);

  \draw[
    bpfdBottom,
    line width=.75pt
  ]
    (SW) -- (SE);

  \draw[
    bpfdBlue,
    line width=.75pt
  ]
    (NW) -- (SW);

  \draw[
    bpfdBlue,
    line width=.75pt
  ]
    (NE) -- (SE);

  \draw[
    bpfdRed,
    line width=.75pt
  ]
    (NW) -- (C) -- (SE);

  \draw[
    bpfdPurple,
    line width=.75pt
  ]
    (SW) -- (C) -- (NE);

  \draw[
    bpfdTop,
    line width=.7pt
  ]
    (2.82,6.18) -- (3.00,6.00) -- (2.82,5.82);

  \draw[
    bpfdBottom,
    line width=.7pt
  ]
    (2.82,.18) -- (3.00,0) -- (2.82,-.18);

  \foreach \d in {-.30,-.10,.10,.30}{
    \draw[
      bpfdBlue,
      line width=.62pt
    ]
      (-.18,{3.18+\d})
      --
      (0,{3+\d})
      --
      (.18,{3.18+\d});

    \draw[
      bpfdBlue,
      line width=.62pt
    ]
      (5.82,{3.18+\d})
      --
      (6,{3+\d})
      --
      (6.18,{3.18+\d});
  }

  \foreach \d in {-.30,-.18,-.06,.06,.18,.30}{
    \draw[
      bpfdRed,
      ->,
      line width=.58pt
    ]
      ({1.68+\d},{4.32-\d}) -- ({1.36+\d},{4.64-\d});
  }

  \foreach \d in {-.10,.10}{
    \draw[
      bpfdRed,
      ->,
      line width=.58pt
    ]
      ({4.68+\d},{1.32-\d}) -- ({4.36+\d},{1.64-\d});
  }

  \foreach \d in {-.14,0,.14}{
    \draw[
      bpfdPurple,
      ->,
      line width=.58pt
    ]
      ({1.68+\d},{1.68+\d}) -- ({1.36+\d},{1.36+\d});
  }

  \foreach \d in {-.24,-.12,0,.12,.24}{
    \draw[
      bpfdPurple,
      ->,
      line width=.58pt
    ]
      ({4.32+\d},{4.32+\d}) -- ({4.64+\d},{4.64+\d});
  }

  \foreach \P in {NW,NE,SW,SE,C}{
    \fill[black] (\P) circle (4.2pt);
  }

  \node[
    anchor=south east
  ] at (-.12,6.20)
    {$v_1$};

  \node[
    anchor=south west
  ] at (6.12,6.20)
    {$v_1$};

  \node[
    anchor=north east
  ] at (-.10,-.16)
    {$v_1$};

  \node[
    anchor=north west
  ] at (6.10,-.16)
    {$v_1$};

  \node[
    anchor=south west
  ] at (3.28,3.05)
    {$v_2$};

  \node at (3.00,6.65)
    {$E_1$};

  \node at (2.55,.62)
    {$E_1$};

  \node[
    anchor=east
  ] at (-.34,3.28)
    {$E_4$};

  \node[
    anchor=west
  ] at (6.34,3.28)
    {$E_4$};

  \node at (1.25,4.40)
    {$E_6$};

  \node at (4.88,4.40)
    {$E_5$};

  \node at (2.05,1.74)
    {$E_3$};

  \node at (4.02,1.62)
    {$E_2$};

  \node at (3.10,4.88)
    {$F_3$};

  \node at (1.72,3.02)
    {$F_4$};

  \node at (4.55,3.02)
    {$F_2$};

  \node at (3.00,1.45)
    {$F_1$};

\end{tikzpicture}

\caption{Fundamental domain for $\faktor{(\mathcal{H}_L \cap \R L)}{L}$. Note the extra vertex in the center.}
\label{fig: bp2funddom}
\end{figure} 

\FloatBarrier

\FloatBarrier

\begin{table}
\caption{Monomial labelings on faces of $\faktor{(\mathcal{H}_L \cap \R L)}{L}$ }
\[\begin{array}{c | c }
\text{Face} & \text{Monomial labeling}\\
\hline
v_1 & 1, \frac{x_1x_2 y_4}{x_4 y_1 y_2}, \frac{x_1x_2^2 x_3 y_5}{x_5 y_1 y_2^2 y_3}, \frac{x_2 x_3 x_4 y_5}{x_5 y_2 y_3 y_4 } \text{ starting from the origin, in counter-clockwise order}\\
\hline
v_2 & \frac{x_2y_5}{x_5y_2} \\
\hline
E_1 & x_1 x_2 y_4 , \frac{x_1 x_2^2 x_3 x_4 y_5 }{x_5 y_2 y_3 } \text{ from bottom to top}\\
\hline
E_2 & \frac{x_1x_2 y_4 y_5}{y_2} \\
\hline
E_3 & x_2y_5 \\
\hline
E_4 & \frac{x_1 x_2^2 x_3 y_4 y_5 }{y_1 y_2}, x_2 x_3 x_4 y_5 \text{ from right to left}\\ \hline
E_5 & \frac{ x_1 x_2^2 x_3 y_5}{x_5 y_2}\\ \hline
E_6 & \frac{x_2 x_3 x_4 y_5 }{x_5 y_2 } \\ \hline
F_1 & x_1 x_2 y_4 y_5 \\ \hline
F_2 & \frac{x_1 x_2^2 x_3 y_4 y_5}{y_2}\\ \hline 
F_3 & \frac{x_1 x_2^2 x_3 x_4 y_5}{x_5 y_2} \\ \hline
F_4 & x_2 x_3 x_4 y_5 \\ \hline 
\end{array} \] 
\label{table: monbp2}
\end{table} 

\FloatBarrier
\begin{figure}[h]
\centering
\begin{tikzpicture}[
  >=stealth,
  line cap=round,
  line join=round,
  object/.style={
    font=\small,
    fill=white,
    inner sep=2.5pt
  },
  morphism/.style={
    ->,
    line width=.62pt
  },
  over/.style={
    preaction={
      draw=white,
      line width=3.2pt,
      -
    }
  },
  degree/.style={
    font=\scriptsize,
    fill=white,
    fill opacity=.96,
    text opacity=1,
    inner sep=1.4pt,
    rounded corners=1pt
  }
]

  \node[object] (E4) at (0,7.2)
    {$\Bbbk\acts E_4=\mathcal{O}(1,1,1)$};

  \node[object] (E3) at (0,5.4)
    {$\Bbbk\acts E_3= \mathcal{O}$};

  \node[object] (E2) at (0,3.6)
    {$\Bbbk\acts E_2= \mathcal{O}(0,-1,0)$};

  \node[object] (E1) at (0,1.8)
    {$\Bbbk\acts E_1= \mathcal{O}(-1,-1,0)$};

  \node[object] (E0) at (0,0)
    {$\Bbbk\acts E_0=\mathcal{O}(-1,-2,-1)$};

  \draw[morphism]
    (E0.north)
    --
    node[
      degree,
      anchor=west,
      xshift=2pt
    ] {$(0,1,1)$}
    (E1.south);

  \draw[morphism]
    (E1.north)
    --
    node[
      degree,
      anchor=west,
      xshift=2pt
    ] {$(1,0,0)$}
    (E2.south);

  \draw[morphism]
    (E2.north)
    --
    node[
      degree,
      anchor=west,
      xshift=2pt
    ] {$(0,1,0)$}
    (E3.south);

  \draw[morphism]
    (E3.north)
    --
    node[
      degree,
      anchor=west,
      xshift=2pt
    ] {$(1,1,1)$}
    (E4.south);

  \draw[morphism]
    (E0.west)
    to[
      out=180,
      in=180,
      looseness=1.20
    ]
    node[
      degree,
      anchor=east,
      xshift=-2pt,
      pos=.50
    ] {$(1,1,1)$}
    (E2.west);

  \draw[morphism,over]
    (E1.west)
    to[
      out=180,
      in=180,
      looseness=1.20
    ]
    node[
      degree,
      anchor=east,
      xshift=-2pt,
      pos=.50
    ] {$(1,1,0)$}
    (E3.west);

  \draw[morphism]
    (E0.west)
    to[
      out=180,
      in=180,
      looseness=1.58
    ]
    node[
      degree,
      anchor=east,
      xshift=-2pt,
      pos=.52
    ] {$(2,3,2)$}
    (E4.west);

  \draw[morphism]
    (E2.east)
    to[
      out=0,
      in=0,
      looseness=1.20
    ]
    node[
      degree,
      anchor=west,
      xshift=2pt,
      pos=.50
    ] {$(1,2,1)$}
    (E4.east);

  \draw[morphism,over]
    (E0.east)
    to[
      out=0,
      in=0,
      looseness=1.32
    ]
    node[
      degree,
      anchor=west,
      xshift=2pt,
      pos=.48
    ] {$(1,2,1)$}
    (E3.east);

  \draw[morphism]
    (E1.east)
    to[
      out=0,
      in=0,
      looseness=1.48
    ]
    node[
      degree,
      anchor=west,
      xshift=2pt,
      pos=.57
    ] {$(2,2,1)$}
    (E4.east);

\end{tikzpicture}
\caption{For $i>j$, an arrow $E_j\to E_i$ labeled
$(a_1,a_2,a_3)$ records
$\operatorname{Ext}^{\bullet}(E_j,E_i)
\cong
H^{\bullet}\!\left(
X_\Sigma,
\mathcal{O}(a_1,a_2,a_3)
\right)$.}
\label{fig:exchoms}
\end{figure}

\sectionnotoc{Acknowledgements} Gabriel Kerr advised this project and first encouraged my interest in toric varieties. I am grateful to Rina Anno and Zongzhu Lin for conversations about algebraic geometry, to Daniel Erman for helpful discussions, and to Christine Berkesch for extensive comments on the manuscript. I also thank an anonymous referee for suggestions that improved the exposition. The 2019 MSRI Summer School on Toric Varieties, organized by David Cox and Hal Schenck, was especially valuable for this work. \\

\sectionnotoc{AI acknowledgments}
ChatGPT 5.6 Pro was used for re-formatting and re-creating images in TeX. Discussions with ChatGPT 5.6 Pro were used to find suitable assumptions needed for the symmetry hypothesis in light of a previously discovered counterexample. 

\newpage
\sectionnotoc{References}
\printbibliography[heading=none]

\end{document}